\documentclass[10pt]{article}
\usepackage{amsmath}
\usepackage{amssymb}
\usepackage{graphicx}
\usepackage{color}
\usepackage{amsthm}
\usepackage{amsfonts}
\usepackage{amscd}
\usepackage{mathrsfs}
\usepackage{bm}
\usepackage{enumerate}
\usepackage{algorithmic, algorithm}

\newcommand{\D}{\displaystyle}

\newcommand{\bx}{\mathbf{x}}
\newcommand{\bn}{\mathbf{n}}
\newcommand{\by}{\mathbf{y}}

\newcommand{\p}{\partial}
\newcommand{\al}{\alpha}

\newcommand{\bga}{\mbox{\boldmath$\gamma$}}

\newcommand{\hk}{R_t(\bx, \by)}
\newcommand{\rhk}{\bar{R}_t(\bx, \by)}
\newcommand{\rrhk}{\bar{\bar{R}}_t(\bx, \by)}

\newcommand{\hkxpj}{R_t(\bx, {\bf p}_j)}

\newcommand{\rhkxpj}{\bar{R}_t(\bx, {\bf p}_j)}

\newcommand{\bfu}{{\bf u}}
\newcommand{\bfp}{{\bf p}}

\newcommand{\bz}{\mathbf{z}}

\newcommand{\M}{{\mathcal M}}

\newcommand{\bV}{\mathbf{V}}
\newcommand{\bA}{\mathbf{A}}

\newcommand{\lm}{\mathcal{M}}
\newcommand{\mathd}{\mathrm{d}}

\newtheorem{theorem}{\textbf{Theorem}}[section]
\newtheorem{lemma}{\textbf{Lemma}}[section]
\newtheorem{remark}{\textbf{Remark}}[section]
\newtheorem{proposition}{\textbf{Proposition}}[section]

\newcommand{\R}{\mathbb{R}}

\numberwithin{equation}{section}

\makeatother

\begin{document}

\title{Convergence of Laplacian Spectra from Point Clouds}

\author{
Zuoqiang Shi%
\thanks{Yau Mathematical Sciences Center, Tsinghua University, Beijing, China,
100084. \textit{Email: zqshi@math.tsinghua.edu.cn.}%
}
\and Jian Sun %
\thanks{Yau Mathematical Sciences Center, Tsinghua University, Beijing, China,
100084. \textit{Email: jsun@math.tsinghua.edu.cn.}%
}
}

 \maketitle

\begin{abstract}
The spectral structure of the Laplacian-Beltrami operator (LBO) on manifolds has been widely used in
many applications, include spectral clustering, dimensionality reduction,
mesh smoothing, compression and editing, shape segmentation, matching and parameterization, and so on.
Typically, the underlying Riemannian manifold is unknown and often given by a set of sample points.
The spectral structure of the LBO is estimated from some discrete Laplace operator
constructed from this set of sample points. In our previous papers \cite{LSS,SS14}, we proposed the
point integral method to discretize the LBO from point clouds, which is also capable to solve
the eigenproblem. Then one fundmental issue is the convergence of the eigensystem of the discrete Laplacian
to that of the LBO. In this paper, for compact manifolds isometrically embedded in
Euclidean spaces possibly with boundary,
we show that the eigenvalues and the eigenvectors obtained by the point integral method
converges to the eigenvalues and the eigenfunctions of the LBO with the Neumann boundary,
and in addition, we give an estimate of the convergence rate.
This result provides a solid mathematical foundation for the point integral method
in the computation of Laplacian spectra from point clouds.
\end{abstract}
\vspace{5mm}

\hspace{4mm}\textbf{keywords:} Laplacian spectra; point integral method; point cloud; convergence rate.


\section{Introduction}
The Laplace-Beltrami operator (LBO) is a fundamental object associated to Riemannian manifolds,
which encodes all intrinsic geometry of the manifolds and has many desirable properties. It is also
related to diffusion and heat equation on the manifold, and is connected to a large body of classical
mathematics (see, e.g.,~\cite{rosenberg1997lrm}).
In recent years, the Laplace-Beltrami operator has attracted much attention in many applied fields.
The spectral structure of the Laplacian-Beltrami operator on manifold has been widely used in
many applications, include spectral clustering, dimension reduction, mesh smoothing, compression,
editing, shape segmentation, matching, parameterization, and so on \cite{belkin2003led, Coifman05geometricdiffusions,OvsjanikovSG08, dgw-spec}.

In this paper, we consider the following eigenproblem on
a smooth, compact $k$-dimensional submanifold $\mathcal{M}$
with the Neumann boundary condition
\begin{equation}
\left\{\begin{array}{rl}
     \Delta_{\mathcal{M}} u(\bx)= \lambda u(x) ,&\bx\in  \mathcal{M} \\
      \frac{\p u}{\p \bn}(\bx)=0,& \bx\in \p \mathcal{M}
\end{array}
\right.
\label{eqn:eigen_neumann}
\end{equation}
where $\Delta_\mathcal{M}$ is the Laplace-Beltrami operator on $\mathcal{M}$.
Let $g$ be the Riemannian metric tensor of $\mathcal{M}$, which is assumed to
be inherited from the ambient space $\R^d$, that is,
$\M$ isometrically embedded in $\R^d$ with the standard Euclidean metric.
If $ \mathcal{M}$ is an open set in $\R^d$, then $\Delta_\mathcal{M}$ becomes
standard Laplace operator, i.e., $\Delta_{ \mathcal{M}} = \sum_{i=1}^d \frac{\p^2 }{\p {x^i}^2}$.

In~\cite{LSS}, we have proposed a numerical method called the point integral method (PIM) to solve the above
eigenproblem on manifolds. In the PIM, we only need a point cloud to discretize the manifold $\M$.
In particular,  we are given a set of points $P=(\bfp_1, \cdots, \bfp_n)$ sampling $\M$.
In addition, we are given a vector $\bV = (V_1, \cdots, V_n)^t$ where $V_i$ is an volume
weight of $\bfp_i$ in $\M$, so that for any Lipschitz
function $f$ on $\M$, $\int_\M f(\bx) d\mu_\bx$ can be approximated by
$\sum_{i=1}^n f(\bfp_i) V_i$. Here $d\mu_\bx$ is the volume form of $\M$.

We consider the following discrete Laplace operator $L_{t, h}$
\begin{equation}
L_{t,h} u(\bfp_i) = \frac{C_t}{t}\sum_{\bfp_j \in P}R\left(\frac{|\bfp_i -\bfp_j|^2}{4t}\right)(u(\bfp_i) - u(\bfp_j))V_j .
\label{eqn:L_th}
\end{equation}
where $R: \R^+ \rightarrow \R^+ $ is a given kernel function and $C_t=\frac{1}{4\pi t}^{k/2}$ is a normalized constant.
If set $V_j= \frac{1}{n}$, the discrete Laplace operator $L_{t, h}$
becomes the well-known weighted graph Laplacian~\cite{belkin2001}.
Denote $\bar{R}$ the primitive function of $-R$, i.e. $\frac{d}{dr}\bar{R}(r)=-R(r)$.
Now we consider the following generalized eigenproblem of $L_{t, h}$
\begin{equation}
-L_{t,h} u(\bfp_i) = \lambda \sum_{p_j \in P} C_t\bar{R}\left(\frac{|\bfp_i -\bfp_j|^2}{4t}\right) u(\bfp_j)V_j
\label{eqn:eigen_dis}
\end{equation}

The purpose of the paper is to show the generalized eigenproblem~\eqref{eqn:eigen_dis}
converges to the eigenproblem~\eqref{eqn:eigen_neumann} at a rate $(t^{1/2} + \frac{h}{t^{k/4+3}})$ for
the eigenvalues and $(t^{1/2} + \frac{h}{t^{k/4+2}})$ for the eigenfunctions, provided the input
data $(P,\bV)$ is an $h$-integral approximation of $\M$ (see Section~\ref{sec:note} for its definition).
Note the rate reported in this paper depends on the dimension $k$ of $\M$ and may not be optimal.

Following~\cite{LSS,SS14}, we bridge the LBO $\Delta_\M$ and the discrete Lapalce operator $L_{t, h}$ using
the following integral Laplace operator
\begin{equation}
  L_tu(\bx)=\frac{C_t}{t}\int_\M R\left(\frac{|\bx -\by|^2}{4t}\right)(u(\bx)-u(\by))\mathd\mu_\by.
\end{equation}
We consider the solution operators associated with $\Delta_\M$, $L_t$ and $L_{t, h}$, and show
all solution operators are in fact compact and additionally the approximation errors
of the solution operators in certain operator norms are bounded (see Theorem~\ref{thm:converge_h1} and \ref{thm:converge_c1}).
This enables us to apply the results from spectral theory to obtain a convergence rate.
Note that it is critical and also common in the numerical analysis to
consider the solution operators~\cite{Boffi10}, instead of the Laplacians themselves
which are not even bounded. Comparing to the result in \cite{SS14},  we improve the error estimations
for both the trancation error and the stability of $L_t$. In particular,
in both Theorem~\ref{thm:integral_error} and~\ref{thm:regularity_boundary},
we bound the approximation errors using the Sobolev norms, instead of the infinite norm
as in~\cite{SS14}.

\subsection{Related work}
The finite element method is popular in solving PDEs on manifolds. Dziuk~\cite{Dziuk88} showed the FEM
converges quadratically in $L^2$ norm and linearly in $H^1$ norm in solving the Poisson equations.
The eigensystem of the LBO computed by the FEM converges linearly~\cite{Strang73, Dodziuk76, Wardetzky06}.
The FEM requires a mesh tesselating the domain and its performace depends heavily
on the quality of meshes. However, it is not an easy task to generate a mesh of
good quality, especially for an irregular domain~\cite{cheng12}, and becomes much more
difficult for a curved submanifold~\cite{BoissonnatDG12}.

We see that the discrete Laplace operator $L_{t, h}$ becomes the weighted graph Laplacian when
the volume weight vector is constant. In the presence of no boundary, Belkin and Niyogi~\cite{CLEM_08}
showed that the spectra of the weighted graph Laplacian converges to
the spectra of $\Delta_\M$. When there is a boundary,  it was
observed in~\cite{Lafon04diffusion, BelkinQWZ12} that
the integral Laplace operator $L_t$ is dominated by the first order derivative and
thus fails to be true Laplacian near the boundary.
Recently, Singer and Wu~\cite{Singer13} showed the spectral
convergence in the presence of the Neumann boundary. In the previous approaches,
the convergence analysis is based on the connection between the weighted graph Laplacian and
the heat operator, and thus it is essential to use the Gaussian kernel in those approaches.
The convergence analysis in this paper is very different from the previous ones.
We consider this convergence problem from the
point of view of solving the Poisson equation on submanifolds, which opens up many tools in the
numerical analysis for studying the graph Laplacian.



The rest of the paper is organized as follows. In Section~\ref{sec:assumption}, we describe the basic assumptions and define
the solutions operators we are working on. The main results are stated in Section~\ref{sec:results}. In Section~\ref{sec:intermediate},
the proofs of the main results are given. The proofs depends on two theorems:
one states that $T_t$ converges to $T$ in $H^1$ norm and the other states that
$T_{t,h}$ converges to $T_t$ in $C^1$ norm, whose proofs are given in Section~\ref{sec:converge-h1} and~\ref{sec:converge-c1} respectively.
Finally, in Section~\ref{sec:conclusion}, we conclude and point out a few direction for future research.

\section{Assumptions and Notations}
\label{sec:assumption}

We follow the assumptions in~\cite{SS14} and use basically the same notations.
\subsection{Assumptions}
\label{sec:assum}
First we assume the function $R: \R^+ \rightarrow \R^+ $ is $C^2$  and satisfies the following conditions:
\begin{itemize}
\item[(i)]
$R(r) = 0$ for $\forall r >1$.
\item[(ii)]
There exists a constant $\delta_0$ so that $R(r)>\delta_0$ for $\forall r<\frac{1}{2}$.
\end{itemize}

Second, we assume both $\M$ and $\p \M$ are compact and $C^{\infty}$ smooth. Consequently,
it is well known that both $\M$ and $\p \M$ have positive reaches.

Finally, we assume the input data $(P,\bV)$ is \textbf{$\mathbf{h}$-integral approximation} of $\M$,
i.e.,
\begin{enumerate}
\item[] For any function $f\in C^1(\M)$, there is a constant $C$
independent of $h$ and $f$ so that
\begin{equation}
\left|\int_\M f(\by) d\mu_{\by} - \sum_{i=1}^n f(\bfp_i)V_i\right| < Ch|\text{supp}(f)|\|f\|_{C^1}, \text{~~and}
\end{equation}
\end{enumerate}
where $\|f\|_{C^1} = \|f\|_\infty +\|\nabla f\|_\infty$ and $|X|$ denotes the volume of $X$.
\begin{remark}
  The quadrature rule we considered in this paper is in a deterministic way. We assume that the sample points $P$
  converge to the underlying manifold $\M$ in the sense that $h\rightarrow 0$. Actually, the methods used in this paper 
  also apply on random samples and the results will be reported in the near future. 
\end{remark}



\subsection{Notations}
\label{sec:note}
To investigate the convergence of the problem~\ref{eqn:eigen_dis} to the problem~\ref{eqn:eigen_neumann}
of the Neumann boundary, we consider the solution operators $T, T_t$ and $T_{t,h}$ as follows.
Denote the operator $T: L^2(\mathcal{M})\rightarrow H^2(\mathcal{M})$ to be the solution operator of
the following problem
\begin{eqnarray}
\label{eqn:neumann}
  \left\{\begin{array}{rl}
      \Delta_\M u(\bx)=f(\bx),&\bx\in \mathcal{M},\\
      \frac{\p u}{\p \bn}(\bx)=0,& \bx\in \p \mathcal{M}.
\end{array}\right.
\end{eqnarray}
where $\bn$ is the out normal vector of $\mathcal{M}$.
Namely, for any $f\in L^2(\mathcal{M})$, $u = T(f)$ with $\int_\M u = 0$ solves the problem \eqref{eqn:neumann}.

Denote $T_t:L^2(\mathcal{M})\rightarrow L^2(\mathcal{M})$ to be the solution operator of the following
problem
\begin{eqnarray}
\label{eqn:integral}
-\frac{1}{t}\int_\mathcal{M} R_t(\bx, \by)(u(\bx)-u(\by)) \mathd \by=\int_{\mathcal{M}} \bar{R}_t(\bx, \by) f(\by) \mathd \by.\quad
\end{eqnarray}
where $R_t(\bx,\by)=C_tR\left(\frac{|\bx -\by|^2}{4t}\right)$ and $\bar{R}_t(\bx,\by)=C_t\bar{R}\left(\frac{|\bx -\by|^2}{4t}\right)$, $\bar{R}(r)=\int_r^{+\infty}R(s)\mathd s$.
Namely,  $u = T_t(f)$ with $\int_\M u = 0$ solves the problem \eqref{eqn:integral};

The last solution operator is $T_{t,h}:C(\mathcal{M})\rightarrow C(\mathcal{M})$ which defined as follows.
\begin{eqnarray}
  T_{t,h}(f)(\bx)=\frac{1}{w_{t,h}(\bx)}\sum_{j}R_t(\bx,\bx_j)u_jV_j-\frac{t}{w_{t,h}(\bx)}\sum_{j}\bar{R}_t(\bx,\bx_j)f(\bx_j)V_j
\end{eqnarray}
where $\D w_{t,h}(\bx)=\sum_{\bx_j\in P}R_t(\bx,\bx_j)V_j$ and $\bfu=(u_1, \cdots, u_n)^t$ with $\D\sum_{i=1}^{n} u_i V_i =0$ solves
the following linear system
\begin{eqnarray}
\label{eqn:dis}
 -\frac{1}{t}\sum_{\bx_j\in P}R_t(\bx_i,\bx_j)(u_i-u_j)V_j=\sum_{\bx_j\in P}\bar{R}_t(\bx_i,\bx_j)f(\bx_j)V_j,\quad \bx_i\in P
\end{eqnarray}
One direct consequence is that $T, T_t$ and $T_{t,h}$ have following properties.
\begin{proposition}
\label{prop:eigen}
For any $t>0$, $h>0$,
  \begin{itemize}
  \item[1. ] $T, T_t$ are compact operators on $H^1(\M)$ into $H^1(\M)$; $T_t, T_{t,h}$ are compact operators on $C^1(\M)$ into $C^1(\M)$.
\item[2. ] All eigenvalues of $T, T_t, T_{t,h}$ are real numbers. All generalized eigenvectors of $T, T_t, T_{t,h}$ are eigenvectors.
  \end{itemize}
\end{proposition}
\begin{proof}
  The proof of (1) is straightforward. First, it is well known that $T$ is compact operator. $T_{t,h}$ is actually finite dimensional operator, so
it is also compact. To show the compactness of $T_t$, we need the following formula,
\begin{eqnarray}
  T_t u = \frac{1}{w_t(\bx)}\int_\M R_t(\bx,\by)T_tu(\by)\mathd \by+ \frac{t}{w_t(\bx)}\int_\M \bar{R}_t(\bx,\by)u(\by)\mathd \by, \quad \forall u\in H^1(\M),
\end{eqnarray}
Using the assumption that $R\in C^2$, direct calculation would gives that
 that $T_tu\in C^2$. This would imply the compactness of $T_t$ both in $H^1$ and $C^1$.

For the operator $T$, the conclusion (2) is well known. The proof of $T_t$ and $T_{t,h}$ are very similar, so here we only present the
proof for $T_t$.

Let $\lambda$ be an eigenvalue of $T_t$ and $u$ is corresponding eigenfunction, then
\begin{eqnarray*}
  L_t T_t u =\lambda L_t u
\end{eqnarray*}
which implies that
\begin{eqnarray*}
  \lambda = \frac{\int_M \int _\M \bar{R}_t(\bx, \by) u^*(\bx)u(\by) \mathd \bx\mathd \by}{\int_\M u^*(\bx)(L_t u)(\bx)\mathd \bx }
\end{eqnarray*}
where $u^*$ is the complex conjugate of $u$.

Using the symmetry of $L_t$ and $\bar{R}(\bx,\by)$, it is easy to show that $\lambda\in \mathbb{R}$.

Let $u$ be a generalized eigenfunction of $T_t$ with multiplicity $m>1$ associate with eigenvalue $\lambda$. Let $v=(T_t-\lambda)^{m-1}u$,
$w=(T_t-\lambda)^{m-2}u$, then $v$
is an eigenfunction of $T_t$ and
\begin{eqnarray*}
  T_t v =\lambda v,\quad (T_t-\lambda)w=v
\end{eqnarray*}
By applying $L_t$ on both sides of above two equations, we have
\begin{eqnarray*}
  \lambda L_t v &=& L_t (T_tv) = \int_\M \bar{R}_t(\bx,\by)v(\by)\mathd \by,\\
L_t v &=& L_t(T_tw)-\lambda L_t w = \int_\M \bar{R}_t(\bx,\by)w(\by)\mathd \by-\lambda L_t w
\end{eqnarray*}
Using above two equations and the fact that $L_t$ is symmetric, we get
\begin{eqnarray*}
  0&=& \left<w,\lambda L_t v - \int_\M \bar{R}_t(\bx,\by)v(\by)\mathd \by\right>_\M\nonumber\\
&=& \left<\lambda L_t w - \int_\M \bar{R}_t(\bx,\by)w(\by)\mathd \by, v\right>_\M\nonumber\\
&=& \left<L_t v, v\right>_\M\ge C\left\|v\right\|_2^2
\end{eqnarray*}
which implies that $(T_t-\lambda)^{m-1}u=v=0$. This proves that $u$ is a generalized eigenfunction of $T_t$ with multiplicity $m-1$.
Repeating above argument, we can show that $u$ is actually an eigenfunction of $T_t$.
\end{proof}

Using the definition of $T, T_t$ and $T_{t,h}$, it is easy to show that the eigen problems $T u =\lambda u$,
$T_{t,h}(u) = \lambda  u$ is
equivalent to the eigen problems \eqref{eqn:eigen_neumann} and \eqref{eqn:eigen_dis} respectively.
Namely their eigenvalues are reciprocal to each other and they share the same eigenspaces.
\begin{proposition}
Let $\theta(u)$ denote the
restriction of $u$ to the sample points $P$, i.e., $\theta(u) = (u(\bfp_1), \cdots, u(\bfp_n))^t$.
\begin{itemize}
\item[1. ] If a function $u$ is an eigenfunction of $T_{t,h}$
with the eigenvalue $\lambda$, then the vector $\theta(u)$ is an eigenvector of the eigenproblem~\eqref{eqn:eigen_dis}
with eigenvalue $1/\lambda$.
\item[2. ] If a vector $\mathbf{u}$ is an eigenvector of the eigenproblem~\eqref{eqn:eigen_dis}
with the eigenvalue $\lambda \neq 0 $, then $I_{\lambda}(\mathbf{u})$ is an
eigenfunction of $T_{t,h}$ with eigenvalue $1/\lambda$, where
\begin{equation*}
I_{\lambda}(\bfu) (\bx) = \frac{ \sum_{p_j \in P} \hkxpj u_j V_j  - \lambda t \sum_{p_j \in P} \rhkxpj u_j V_j} {\sum_{p_j \in P} \hkxpj V_j }.
\end{equation*}
\item[3. ]
A function $u$ is the eigenfunction of the eigenproblem~\eqref{eqn:eigen_neumann}
with the eigenvalue $\lambda \neq 0$
if and only if the function $u$ is an eigenfunction of $T $
with the eigenvalue $1 / \lambda$.
\end{itemize}
\label{prop:eigen_intergral_dis}
\end{proposition}
Using the above proposition, we only need to prove the eigenvalues and the eigenfunctions 
of $T_{t,h}$ converge to the eigenvalues and the eigenfunctions of $T$.

\section{Main Results}
\label{sec:results}
Let $X$ be a complex Banach space and $L:X\rightarrow X$ be a compact linear operator. The resolvent set $\rho(L)$ is given by the complex numbers $z\in \mathbb{C}$
such that $z-L$ is bijective. The spectrum of $L$ is $\sigma(L)=\mathbb{C}\backslash\rho(L)$. It is well known that $\sigma(L)$ is a countable set with
no limit points other than zero. All non-zero value s in $\sigma(L)$ are eigenvalues.
If $\lambda$ is a nonzero eigenvalue of $L$, the ascent multiplicity $\al$ of $\lambda-L$ is the smallest integer such that $\ker(\lambda-L)^\al=\ker(\lambda-L)^{\al+1}$.

Given a closed smooth curve $\Gamma\subset \rho(L)$ which encloses the eigenvalue $\lambda$ and no other elements of $\sigma(L)$, the Riesz spectral projection associated with
$\lambda$ is defined by
\begin{eqnarray}
  E(\lambda, L)=\frac{1}{2\pi i}\int_{\Gamma} (z-L)^{-1}\mathd z,
\end{eqnarray}
where $i=\sqrt{-1}$ is the unit imaginary. The definition does not depend on the chosen of $\Gamma$.
It is well known that $E(\lambda,L):X\rightarrow X$ has following properties:
\begin{itemize}
\item[1. ] $E(\lambda,L)\circ E(\lambda,L)=E(\lambda,L)$,\quad $L\circ E(\lambda,L)=E(\lambda,L)\circ L$, $E(\lambda,L)\circ E(\mu,L)=0,\;\text{if}\; \lambda\ne \mu$.
\item[2. ] $E(\lambda,L)X=\ker (\lambda-L)^\al$, where $\al$ is the ascent multiplicity of $\lambda-L$.
\item[3. ] If $\Gamma\subset \rho(L)$ encloses more eigenvalues $\lambda_1,\cdots,\lambda_n$, then
$$\D E(\lambda_1,\cdots,\lambda_n,L)X=\oplus_{i=1}^n\ker (\lambda_1-L)^{\al_i}$$
 where $\al_i$ is the ascent multiplicity of $\lambda_i-L$.
\end{itemize}
The properties (2) and (3) are of fundamental importance for the study of eigenvector approximation.

It is well-known that the eigenvalues of both $T$ can be sorted as
\begin{eqnarray}
  \lambda_1\le \lambda_2\le \cdots,\le \lambda_n\le \cdots < 0,\nonumber
\end{eqnarray}
where the same eigenvalue are repeated accroding to its multiplicity.

Now, we are ready to state the main theorems.
\begin{theorem}
Assume the submanifold $\M$ and $\p \M$ are $C^\infty$ smooth and the input data $(P,\bV)$ is an $h$-integral approximation of $\M$,
Let $\lambda_i$ be the $i$th smallest eigenvalue of $T$ counting multiplicity, and $\lambda_{i}^{t,h}$ be the $i$th smallest eigenvalue
of $T_{t,h}$ counting multiplicity, then there exists a constant $C$ such that
$$|\lambda_i^{t,h} - \lambda_i|\le C  \left(t^{1/2}+\frac{h}{t^{k/4+3}}\right) , $$
and there exist another constant $C$ such that, for any $\phi\in E(\lambda_i,T)X$ and $X=H^1(\M)$,
$$\|\phi-E(\sigma_{i}^{t,h},T_{t,h})\phi\|_{H^1(\M)} \le C \left(t^{1/2}+\frac{h}{t^{k/4+2}}\right) .$$
where $\sigma_i^{t,h}=\{\lambda_j^{t,h}\in \sigma(T_{t,h}): j\in I_i\}$ and $I_i=\{j\in \mathbb{N}: \lambda_j=\lambda_i\}$.
\label{thm:eigen_neumann}
\end{theorem}
Using the above theorem, we know that the eigenvalues of $T_{t,h}$ converge to
the eigenvalues of $T$. From Proposition \ref{prop:eigen_intergral_dis}, the
inverses of the eigenvalues of $T_{t,h}$ and $T$ are the eigenvalues of
\eqref{eqn:eigen_dis} and \eqref{eqn:eigen_neumann} respectively. Then, the
eigenvalues of \eqref{eqn:eigen_dis} converge to the eigenvalues of
\eqref{eqn:eigen_neumann}. The above theorem also states the convergence of the
generalized eigenfunctions, namely the generalized eigenfunctions of $T_{t,h}$
to the generalized eigenfunctions of $T$. 
Since all  generalized eigenfunctions of $T_{t,h}$ and $T$ 
are indeed the eigenfunctions as proved in Propositions \ref{prop:eigen}, 
we obtain the convergence of the eigenfunctions of $T_{t,h}$ to that of $T$.
Then the convergence of the eigenfunctions of \eqref{eqn:eigen_dis}
to that of \eqref{eqn:eigen_neumann} can be obtained by using the facts 
that $T_{t,h}$ and \eqref{eqn:eigen_dis} ($T$ and \eqref{eqn:eigen_neumann})
share the same eigenfunctions as proved in 
Proposition \ref{prop:eigen_intergral_dis}.



\section{Structure of the Proof}
\label{sec:intermediate}
To prove Theorem \ref{thm:eigen_neumann}, we need some results in
perturbation theory for spectral projection. First, we need the following theorem
to obtain the convergence rate of the eigenvalues.
\begin{theorem}
\label{thm:resovent}
Let $(X,\|\cdot\|_X)$ be an arbitrary Banach space. Let $S$ and
$T$ be compact linear operators on $X$ into $X$.
 Let $z\in \rho(T)$. Assume
\begin{eqnarray}
  \|T-S\|_{X} \le \frac{1}{2\|(z-T)^{-1}\|_{X}}.
\end{eqnarray}
Then $z\in \rho(S)$ and $(z-S)^{-1}$ has the bound
\begin{eqnarray}
  \|(z-S)^{-1}\|_{X}\le 2\|(z-T)^{-1}\|_{X}.
\end{eqnarray}
\end{theorem}
\begin{proof}For any $x\in X$
  \begin{eqnarray}
    \|(z-S)x\|_X&\ge& \|(z-T)x\|_X-\|(T-S)x\|_X\nonumber\\
&\ge& \left(\|(z-T)^{-1}\|_{X}-\|T-S\|_X\right)\|x\|_X\nonumber\\
&\ge&\frac{1}{2}\|(z-T)^{-1}\|_{X}\|x\|_X
  \end{eqnarray}
Using this inequality and the assumption that $S$ is compact operator, we have $z\in \rho(S)$ and
\begin{eqnarray}
  \|(z-S)^{-1}\|_{X}\le 2\|(z-T)^{-1}\|_{X}.
\end{eqnarray}
\end{proof}
We also need the following theorem (e.g.~\cite{Atkinson67}) to get the convergence rate of the eigenfunctions.
\begin{theorem}
\label{thm:converge_evector}
  Let $(X,\|\cdot\|_X)$ be an arbitrary Banach space. Let $S$ and
$T$ be compact linear operators on $X$ into $X$. Let $z_0\in \mathbb{C}, \; z_0\ne 0$ and let $\epsilon>0$ be less than $|z_0|$,
denote the circumference $|z - z_0|=\epsilon$ by $\Gamma$ and assume $\Gamma\subset \rho(T)$.
Denote the interior of $\Gamma$ by $U$. Let $\sigma_T=U\cap \sigma(T)\ne \emptyset$. $\sigma_S=U\cap\sigma(S)$.
 Let $E(\sigma_S,S)$ and $E(\sigma_T,T)$ be the corresponding spectral projections of $S$ for $\sigma_S$ and $T$ for $\sigma_T$, i.e.
 \begin{eqnarray}
   E(\sigma_S,S)=\frac{1}{2\pi i}\int_{\Gamma}(z-S)^{-1}\mathd z,\quad    E(\sigma_T,T)=\frac{1}{2\pi i}\int_{\Gamma}(z-T)^{-1}\mathd z.
 \end{eqnarray}
Assume
\begin{eqnarray}
 \|(T-S)S\|_X\le \min_{z\in \Gamma} \frac{|z|}{\|(z-T)^{-1}\|_X}
\end{eqnarray}
Then, we have
\begin{itemize}
\item[(1).] Dimension $E(\sigma_S,S)X = E(\sigma_T,T)X$, thereby $\sigma_S$ is nonempty and of the same multiplicity as $\sigma_T$.
\item[(2).] For every $x\in E(\sigma_T,T) X$,
\begin{eqnarray}
  \|x-E(\sigma_S,S)x\|_X \le \frac{M\epsilon}{c_0}\left(\|(T-S)x\|_X+\|x\|_X\|(T-S)S\|_X\right).
\end{eqnarray}
where $M=\max_{z\in \Gamma}\|(z-T)^{-1}\|_X$, $c_0=\min_{z\in \Gamma}|z|$.
\end{itemize}
\label{thm:eigen_perturb}
\end{theorem}

In order to utilize the above two theorems from pertubation theory, we show the following three theorems
which bound the approximation errors of the solution operators $T, T_{t}, T_{t, h}$ in certain operator norms.
\begin{theorem}
\label{thm:converge_h1}
Under the assumptions in \ref{sec:assum},
there exist generic constants $C>0$ and $T_0>0$ only depend on $\M$, such that
  \begin{eqnarray}
    \|T-T_t\|_{H^1}\le Ct^{1/2}
  \end{eqnarray}
provides $t\le T_0$.
\end{theorem}

\begin{theorem}
\label{thm:converge_c1}
Under the assumptions in \ref{sec:assum}, there exist generic constants $C>0$ and $T_0>0$ only depend on $\M$,
 such that
  \begin{eqnarray}
    \|T_{t,h}-T_t\|_{C^1}\le \frac{Ch}{t^{k/4+2}}
  \end{eqnarray}
as long as $t\le T_0$ and $\frac{h}{\sqrt{t}}\le T_0$.
\end{theorem}
\begin{theorem}
\label{thm:bound} Under the assumptions in \ref{sec:assum}, there exist generic constants $C>0$ and $T_0>0$ only depend on $\M$, such that
  \begin{eqnarray}
    \|T_t\|_{H^1}\le C,\quad \|T_{t,h}\|_{\infty}\le Ct^{-k/4}, \quad
\|T_{t,h}\|_{C^1}\le Ct^{-(k+2)/4}
  \end{eqnarray}
as long as $t\le T_0$ and $\frac{h}{\sqrt{t}}\le T_0$.
\end{theorem}
\noindent 
The above three theorems will be proved in Section \ref{sec:converge-h1}, Section \ref{sec:converge-c1} and Section \ref{sec:bound} respectively.

Before we prove Theorem~\ref{thm:eigen_neumann}, we show the the estimates of $\|(z-T)^{-1}\|_{H^1(\M)}$
and $\|(z-T_t)^{-1}\|_{C^1(\M)}$ in the following two lemmas, which are needed to invoke
the results from pertubation theory. 
\begin{lemma}
\label{lem:norm_inverse_h1}
 Let $T$ be the solution operator of
the Neumann problem \eqref{eqn:neumann} and $z\in \rho(T)$, then
  \begin{eqnarray}
    \|(z-T)^{-1}\|_{H^1(\M)}\le \max_{n\in \mathbb{N}} \frac{1}{|z-\lambda_n|},
  \end{eqnarray}
where $\{\lambda_n\}_{n\in \mathbb{N}}$ is the set of eigenvalues of $T$.
\end{lemma}
\begin{proof}
Suppose $\phi_j, \;j\in \mathbb{N}$ be the normalized eigenfunction of $T$ corresponding to $\lambda_j, \;j\in \mathbb{N}$.
Then it is well known that $\{\phi_j\}_{j\in \mathbb{N}}$ is a orthonormal basis of $H^1(\M)$.

For any $x\in H^1(\M)$, $z\in \rho(T)$, first we can expand $x$ over $\{\phi_j\}_{j\in \mathbb{N}}$ to obtain
\begin{eqnarray}
  x=\sum_{j=1}^\infty c_j\phi_j.
\end{eqnarray}
Then, we have
  \begin{eqnarray}
    \|(z-T)x\|_{H^1}&=& \left\|\sum_{j=1}^\infty c_j (z-T)\phi_j\right\|_{H^1}
=\left\|\sum_{j=1}^\infty c_j (z-\lambda_j)\phi_j\right\|_{H^1}\nonumber\\
&=&\left(\sum_{j=1}^\infty c_j^2 |z-\lambda_j|^2\right)^{1/2} \ge
\min_{n\in \mathbb{N}}|z-\lambda_n| \left(\sum_{j=1}^\infty c_j^2\right)^{1/2}\nonumber\\
&=& \min_{n\in \mathbb{N}}|z-\lambda_n|\|x\|_{H^1}
  \end{eqnarray}
\end{proof}

\begin{lemma}
\label{lem:norm_inverse_c1}
Let $T_t$ be the solution operator of the integral equation \eqref{eqn:integral}. For any $z\in \mathbb{C}\backslash \bigcup_{n\in \mathbb{N}}B(\lambda_n, r_0)$ 
with $r_0>\|T-T_t\|_{H^1}$, then
  \begin{eqnarray}
   \|(z-T_t)^{-1}\|_{C^1}\le \max\left\{\frac{2|\M|}{ |z|t^{(k+2)/4}} \left(\min_{n\in \mathbb{N}}|z-\lambda_n| - \|T-T_t\|_{H^1}\right)^{-1},\frac{2}{|z|}\right\}
\end{eqnarray}
\end{lemma}
\begin{proof}For any $x\in H^1(\M)$,
  \begin{eqnarray}
    \|(z-T_t)x\|_{H^1}&\ge& \|(z-T)x\|_{H^1}-\|(T-T_t)x\|_{H^1}\nonumber\\
&\ge&\left(\min_{n\in \mathbb{N}}|z-\lambda_n| - \|T-T_t\|_{H^1}\right)\|x\|_{H^1}
  \end{eqnarray}
Then $(z-T_t)^{-1}$ exists and
\begin{eqnarray}
  \|(z-T_t)^{-1}\|_{H^1}\le \left(\min_{n\in \mathbb{N}}|z-\lambda_n| - \|T-T_t\|_{H^1}\right)^{-1}.
\end{eqnarray}
For any $u\in C^1(\M)$,
\begin{eqnarray}
  \|(z-T_t)^{-1}u\|_{H^1}\le \left(\min_{n\in \mathbb{N}}|z-\lambda_n| - \|T-T_t\|_{H^1}\right)^{-1} |\M|\|u\|_{C^1}
\end{eqnarray}
where $|\M|$ is the volume of the manifold $\M$.

On the other hand, let $v=(z-T_t)^{-1}u$ which means $v=(u+T_tv)/z$
\begin{eqnarray}
  \|v\|_{C^1}&\le& \frac{1}{|z|}\left(\|u\|_{C^1}+\|T_t v\|_{C^1}\right)\nonumber\\
&\le & \frac{1}{|z|}\left(\|u\|_{C^1}+t^{-(k+2)/4}\|v\|_{L^2}\right)\nonumber\\
&\le & \frac{1}{|z|}\left(\frac{|\M|}{ t^{(k+2)/4}} \left(\min_{n\in \mathbb{N}}|z-\lambda_n| - \|T-T_t\|_{H^1}\right)^{-1}+1\right)\|u\|_{C^1}\nonumber
\end{eqnarray}
which proves that
\begin{eqnarray}
   \|(z-T_t)^{-1}\|_{C^1}\le \max\left(\frac{2|\M|}{|z| t^{(k+2)/4}}
\left(\min_{n\in \mathbb{N}}|z-\lambda_n| - \|T-T_t\|_{H^1}\right)^{-1},\frac{2}{|z|}\right)
\end{eqnarray}
\end{proof}

Using the above theorems and lemmas, we show that $\sigma(T_t)$
is close to $\sigma(T)$, which will be used in the proof of Theorem~\ref{thm:eigen_neumann}.
\begin{theorem}
Let $T_t$ be the solution operator of the integral equation \eqref{eqn:integral}, then
\label{thm:converge_ev_Tt}
  \begin{eqnarray}
    \sigma(T_t)\subset \bigcup_{n\in \mathbb{N}}B\left(\lambda_n,2\|T-T_t\|_{H^1(\M)}\right)
  \end{eqnarray}
\end{theorem}
\begin{proof}
  Let $r_0=\|T-T_t\|_{H^1(\M)}$, $\mathcal{A}=\mathbb{C}\backslash \bigcup_{n\in \mathbb{N}}B(\lambda_n, 2r_0)$.
For any $z\in \mathcal{A}$, using Lemma \ref{thm:resovent}, we have
\begin{eqnarray*}
  \|(z-T)^{-1}\|_{H^1(\M)}\le \max_{n\in \mathbb{N}} \frac{1}{|z-\lambda_n|}\le \frac{1}{2r_0}
\end{eqnarray*}
which implies that
\begin{eqnarray*}
\|T-T_t\|_{H^1(\M)}=r_0 \le \frac{1}{2\|(z-T)^{-1}\|_{H^1(\M)}}.
\end{eqnarray*}
Then using Theorem \ref{thm:resovent}, we have $z\in \rho(T_t)$.

Since $z$ is arbitrary in $\mathcal{A}$, we get $\mathcal{A}\subset \rho(T_t)$. This means that
\begin{eqnarray*}
  \sigma(T_t)=\mathbb{C}\backslash\rho(T_t)\subset \mathbb{C}\backslash \mathcal{A}=
\bigcup_{n\in \mathbb{N}}B(\lambda_n, 2\|T-T_t\|_{H^1(\M)}).
\end{eqnarray*}
\end{proof}

Now we can show Theorem~\ref{thm:eigen_neumann}, i.e., the convergence of the eigenproblem with the Neumann boundary .
\begin{proof}
 Let $r_1=\frac{4|\M|}{t^{k/4+1}}\left\|T_{t,h}-T_t\right\|_{C^1}+\|T-T_t\|_{H^1(\M)}$, $\mathcal{A}=\mathbb{C}\backslash \left[\bigcup_{n\in \mathbb{N}}B(\lambda_n, r_1)
\bigcup B(0,t^{1/2})\right]$.
For any $z\in \mathcal{A}$, using Lemma \ref{thm:resovent}, we have
\begin{eqnarray*}
  \|(z-T_{t})^{-1}\|_{C^1}&\le& \frac{2|\M|}{|z| t^{(k+2)/4}} \left(\min_{n\in \mathbb{N}}|z-\lambda_n| - \|T-T_t\|_{H^1}\right)^{-1}\nonumber\\
&\le & \frac{2|\M|}{ t^{k/4+1}} \left(r_1 - \|T-T_t\|_{H^1}\right)^{-1}\nonumber\\
&=& \left(2\left\|T_{t,h}-T_t\right\|_{C^1}\right)^{-1}
\end{eqnarray*}
or 
\begin{eqnarray*}
  \|(z-T_{t})^{-1}\|_{C^1}&\le& \frac{2}{|z|}\le \frac{2}{t^{1/2}}\le \left(2\left\|T_{t,h}-T_t\right\|_{C^1}\right)^{-1}
\end{eqnarray*}
as long as $h$ small enough by using Theorem \ref{thm:converge_c1}.

Both above two inequalies implies that
\begin{eqnarray*}
\left\|T_{t,h}-T_t\right\|_{C^1}\le \frac{1}{2\|(z-T_t)^{-1}\|_{C^1}}.
\end{eqnarray*}
Then using Lemma \ref{thm:resovent}, we have $z\in \rho(T_{t,h})$.

Since $z$ is arbitrary in $\mathcal{A}$, we get $\mathcal{A}\subset \rho(T_{t,h})$. This means that
\begin{eqnarray}
\label{eqn:est_eigen}
  \sigma(T_{t,h})=\mathbb{C}\backslash\rho(T_{t,h})\subset \mathbb{C}\backslash \mathcal{A}
= \bigcup_{n\in \mathbb{N}}B(\lambda_n, r_1)\bigcup B(0,t^{1/2}).
\end{eqnarray}
Moreover, using Theorem \ref{thm:converge_ev_Tt} and the definition of $r_1$, we have
 \begin{eqnarray}
\label{eqn:est_eigen_T}
  \sigma(T_{t})\subset \bigcup_{n\in \mathbb{N}}B(\lambda_n, 2r_1).
\end{eqnarray}
On the other hand, using Theorem \ref{thm:converge_h1} and \ref{thm:converge_c1}, we know that there exist $C>0$ independent on $t$ and $h$, such that
\begin{eqnarray}
\label{eqn:est_radius}
  r_1\le C\left(t^{1/2}+\frac{h}{t^{k/2+3}}\right).
\end{eqnarray}
For any fixed eigenvalue $\lambda_n\in \sigma(T)$, let $\D\gamma_j=\min_{\lambda\in \sigma(T)\backslash \{\lambda_j\}}
|\lambda_j-\lambda|, \;j\in \mathbb{N}$ and
$\D\gamma=\min_{j\le n}\gamma_j$. Using the structure of $\sigma(T)$, we know that $\gamma>0$. Without loss of generality,
we can assume $t,h$ are small enough such that $r_1\le \gamma/6$.

Let $\Gamma_j=\{z\in \mathbb{C}:|z-\lambda_j|=\gamma/3\}$, $U_j$ be the aera enclosed by $\Gamma_j$.
 Let
$$\sigma_{t,j}=\sigma(T_t)\bigcap U_j,\quad\sigma_{t,h,j}=\sigma(T_{t,h})\bigcap U_j.$$

Using the definition of $\Gamma_j$, we know for any $j\le n$,
$\Gamma_j\subset \rho(T), \rho(T_t)$ and $ \rho(T_{t,h})$.

In order to apply Theorem \ref{thm:converge_evector}, we need to verify the conditions
\begin{eqnarray}
\label{eqn:cond_T}
  \|(T-T_t)T_t\|_{H^1}&\le& \min_{z\in \Gamma_j}\frac{|z|}{\|(z-T)^{-1}\|_{H^1}},\\
\label{eqn:cond_Tt}
\|(T_t-T_{t,h})T_{t,h}\|_{C^1}&\le& \min_{z\in \Gamma_j}\frac{|z|}{\|(z-T_t)^{-1}\|_{C^1}}.
\end{eqnarray}
Using Lemma \ref{lem:norm_inverse_h1} and the choice of $\Gamma_j$, we have
\begin{eqnarray}
  \min_{z\in \Gamma_j}\frac{|z|}{\|(z-T)^{-1}\|_{H^1}}\ge \frac{\min_{z\in \Gamma_j}|z|}{\max_{z\in \Gamma_j}\|(z-T)^{-1}\|_{H^1}}
\ge (|\lambda_j|-\gamma/3)\min_{z\in \Gamma_j,n\in \mathbb{N}}|z-\lambda_n|=(|\lambda_j|-\gamma/3)\gamma/3.\nonumber
\end{eqnarray}
Then, using Theorem \ref{thm:converge_h1} and Lemma \ref{thm:bound},
condition \ref{eqn:cond_T} is true as long as $t$ is small enough.

Using Lemma \ref{lem:norm_inverse_c1}, we have
\begin{eqnarray}
  \min_{z\in \Gamma_j}\frac{|z|}{\|(z-T_t)^{-1}\|_{C^1}}&\ge& \frac{\min_{z\in \Gamma_j}|z|}{\max_{z\in \Gamma_j}\|(z-T_t)^{-1}\|_{C^1}}
\nonumber\\
&\ge& \frac{(|\lambda_j|-\gamma/3)^2t^{(k+2/4)}}{2|\mathcal{M}|}
\left(\min_{z\in \Gamma_j,n\in \mathbb{N}}|z-\lambda_n|-\|T-T_t\|_{H^1}\right)\nonumber\\
&\ge&\frac{(|\lambda_j|-\gamma/3)^2t^{(k+2/4)}\gamma}{12|\mathcal{M}|}.
\end{eqnarray}
To get the last inequality, we use the fact that $\|T-T_t\|_{H^1}<r_1\le \gamma/6$ and
$\D\min_{z\in \Gamma_j,n\in \mathbb{N}}|z-\lambda_n|=\gamma/3$.

Using Theorem \ref{thm:converge_c1} and Lemma \ref{thm:bound},
 we can choose $h$ small enough such that condition \ref{eqn:cond_Tt} is satisfied.

Then using Theorem \ref{thm:converge_evector}, we have
\begin{eqnarray}
  \dim(E(\lambda_j,T))=  \dim(E(\sigma_{t,j},T_t)) =  \dim(E(\sigma_{t,h,j},T_{t,h})),\quad j=1,\cdots,n.
\end{eqnarray}
Combining \eqref{eqn:est_eigen}, above equality would imply that
\begin{eqnarray}
  |\lambda_j^{t,h}-\lambda_j|\le C\left(t^{1/2}+\frac{h}{t^{k/2+3}}\right),\quad j\in \mathbb{N}.
\end{eqnarray}

The convergence of eigenspace is also given by Theorem \ref{thm:converge_evector}.
For any
$x\in E(\lambda_n,T)$, $\|x\|_{C^1}=1$,
\begin{eqnarray}
  \|x-E(\sigma_{t,n},T_t)x\|_{H^1}&\le& C\|T-T_t\|_{H^1}\|x\|_{H^1}\le C t^{1/2},\\
  \|E(\sigma_{t,n},T_t)x-E(\sigma_{t,h,n},T_{t,h})x\|_{C^1}&\le& C\left(\|T_t-T_{t,h}\|_{C^1}+\|(T_t-T_{t,h})T_{t,h}\|_{C^1}\right) \le \frac{Ch}{t^{k/4+2}}.
\end{eqnarray}
Finally, we have
\begin{eqnarray}
  \|x-E(\sigma_{t,h,n},T_{t,h})x\|_{H^1}\le C \left(t^{1/2}+\frac{h}{t^{k/4+2}}\right).
\end{eqnarray}
\end{proof}



\section{Proof of Theorem \ref{thm:converge_h1}}
\label{sec:converge-h1}
We first introduce the local coordinate of the manifold $\M$.
According to Proposition 6.1 of \cite{SS14}, $\M$ can be locally parametrized
as follows.
\begin{eqnarray}
  \bx=\Phi(\bga): \Omega \subset \mathbb{R}^k \rightarrow \lm \subset \R^d
\end{eqnarray}
where $\bga=(\gamma^1, \cdots, \gamma^k)^t \in \R^k$ and $\bx = (x^1, \cdots, x^d)^t\in M$.


Let $\p_{i'}=\frac{\p}{\p\gamma^{i'}}$ be the tangent vector along the direction $\gamma^{i'}$.
Since $\M$ is a submanifold in $\R^d$ with induced metric, $\p_{i'} = (\p_{i'} \Phi^1, \cdots, \p_{i'} \Phi^d)$
and the metric tensor $$g_{i'j'} = <\p_{i'}, \p_{j'}> = \p_{i'} \Phi^l \p_{j'} \Phi^l.$$
Let $g^{i'j'}$ denote the inverse of $g_{i'j'}$, i.e.,
\begin{eqnarray*}
g_{i'l'}\,g^{l'j'}=\delta_{ij}=\left\{\begin{array}{cc}
1,&i=j,\\
0,&i\ne j.
\end{array}\right.
\end{eqnarray*}
For any function $f$ on $\M$, $\nabla f = g^{i'j'}\p_{j'} f \p_{i'}$ denotes the gradient
of $f$. For convenience, let $\nabla^j f$ denote the $x^j$ component of the gradient $\nabla f$,
i.e.,
\begin{eqnarray}
\label{relation-gp}
  \nabla^j f=\p_{i'}\Phi^jg^{i'j'}\p_{j'} f \quad \text{and} \quad \p_{i'}f =\p_{i'}\Phi^j \nabla^jf.
\end{eqnarray}

Then Theorem \ref{thm:converge_h1} is an easy corollary of following three theorems.
\begin{theorem}
Assume $\M$ and $\p\M$ are $C^\infty$. Let $u(\bx)$ be the solution of the problem~\eqref{eqn:neumann} and $u_t(\bx)$ be the solution of
corresponding problem~~\eqref{eqn:integral}. If $f\in C^{\infty}(\M), g\in C^{\infty}(\p\M)$ in both problems, then there exists constants
$C, T_0$ depending only on $\M$ and $\p\M$, so that for any $t\le T_0$,
\begin{eqnarray}
\label{eqn:integral_error_l2}
\left\|L_t (u- u_t)-\int_{\p\M}n_j(\by)\eta_i(\bx,\by)\nabla_i\nabla_ju(\by)\rhk\mathd \tau_\by\right\|_{L^2(\M)}&\le& Ct^{1/2}\|u\|_{H^3(\mathcal{M})},\\
\label{eqn:integral_error_h1}
\left\|\nabla \left(L_t (u- u_t)-\int_{\p\M}n_j(\by)\eta_i(\bx,\by)\nabla_i\nabla_ju(\by)\rhk\mathd \tau_\by\right)\right\|_{L^2(\M)}
&\leq& C\|u\|_{H^3(\mathcal{M})}.
\end{eqnarray}
where $\eta(\bx,\by)=\xi^i(\bx,\by)\p_i\Phi(\al)$ and $\al=\Phi^{-1}(\bx), \xi(\bx,\by)=\Phi^{-1}(\bx)-\Phi^{-1}(\by)$.
\label{thm:integral_error}
\end{theorem}

\begin{theorem}
Assume $\M$ and $\p\M$ are $C^\infty$.
Let $u(\bx)$ solves the integral equation
\begin{eqnarray}
  L_t u = r(\bx)-\bar{r}
\end{eqnarray}
where $r\in H^1(\M)$ and $\bar{r}=\frac{1}{|\M|}\int_\M r(\bx)\mathd\bx$.
Then, there exist a constant $C>0, T_0>0$ independent on $t$, such that
\begin{eqnarray}
  \|u\|_{H^1(\M)}\le C\left(\|r\|_{L^2(\M)}+t\|\nabla r\|_{L^2(\M)}\right)
\end{eqnarray}
as long as $t\le T_0$.
\label{thm:regularity}
\end{theorem}

\begin{theorem}
Assume $\M$ and $\p\M$ are $C^\infty$.
Let $u(\bx)$ solves the integral equation
\begin{eqnarray}
  L_t u = \int_{\p\M}b_i(\by)\eta_i(\bx,\by)\rhk\mathd \tau_\by-\bar{b}
\end{eqnarray}
where $\eta$ is same as that in Theorem \ref{thm:integral_error} and
\begin{eqnarray*}
  \bar{b}=\frac{1}{|\M|}\int_\M \int_{\p\M}b_i(\by)\eta_i(\bx,\by)\rhk\mathd \tau_\by\mathd\bx.
\end{eqnarray*}

Then, there exist constant $C>0, T_0>0$ independent on $t$, such that
\begin{eqnarray}
  \|u\|_{H^1(\M)}\le C\sqrt{t}\;\|\mathbf{b}\|_{H^1(\M)}.
\end{eqnarray}
as long as $t\le T_0$.
\label{thm:regularity_boundary}
\end{theorem}

The proof of Theorem \ref{thm:regularity} can be found in \cite{SS14}. In next two subsections,
we will prove Theorem \ref{thm:regularity_boundary} and Theorem \ref{thm:integral_error} sequentially.

\subsection{Proof of Theorem \ref{thm:regularity_boundary}}
First, we need following four lemmas which have been proved in \cite{SS14}.
\begin{lemma} For any function $u\in L^2(\mathcal{M})$, 
there exist a constant $C>0$ independent on $t$ and $u$, such that
  \begin{eqnarray}
    \left<u,L_t u\right>_\M \ge C\int_\mathcal{M} |\nabla v|^2
\mathd \mu_\bx
  \end{eqnarray}
where $\left<f, g\right>_{\mathcal{M}} = \int_{\mathcal{M}} f(\bx)g(\bx)\mathd\mu_\bx$ for any $f, g\in L_2(\mathcal{M})$ and
\begin{eqnarray}
v(\bx)=\frac{C_t}{w_t(\bx)}\int_{\mathcal{M}}R\left(\frac{|\bx-\by|^2}{4t}\right) u(\by)\mathd \mu_\by, 
\label{eqn:smooth_v}
\end{eqnarray}
and $w_t(\bx) = C_t\int_{\mathcal{M}}R\left(\frac{|\bx-\by|^2}{4t}\right)\mathd \mu_\by$.
\label{lem:elliptic_v}
\end{lemma}
\begin{lemma}
Assume $\M$ and $\p\M$ are $C^\infty$. There exist a constant $C>0$ independent on $t$
so that for any function $u\in L_2(\M)$ with $\int_\M u = 0$ and for any sufficient small $t$
\begin{eqnarray}
\left<L_tu, u\right>_{\mathcal{M}} \geq C\|u\|_{L_2(\M)}^2
\end{eqnarray}
\label{lem:elliptic_L_t}
\end{lemma} 

\begin{lemma}
Assume both $\M$ and $\p \M$ are $C^\infty$ smooth. For sufficiently small $t$ and any $\bx\in \M$, 
there are constants $w_{\min}>0, w_{\max}>0$ depending only on the geometry of $\M, \p\M$,  
so that $$w_{\min} \leq \int_\M R_t(\bx, \by) \mathd \mu_\by \leq w_{\max}$$ 
\label{lem:bound_int_R_t}
\end{lemma}

\begin{lemma}
Assume both $\M$ and $\p \M$ are $C^\infty$ smooth. For sufficiently small $t$ and any $\bx\in \R^d$, 
there is a constant $C$ depending only on the geometry of $\M, \p\M$, so that
for any integer $s\geq 0$ and $K_t(\bx,\by)=R_t(\bx, \by), \bar{R}_t(\bx, \by), \bar{\bar{R}}_t(\bx, \by)$, 
\begin{eqnarray*}
&&\int_{\M} K_t(\bx, \by) |\bx-\by|^s \mathd \mu_\by \leq Ct^{s/2}, \quad 
\int_{\p\M} K_t(\bx, \by) |\bx-\by|^s \mathd \mu_\by \leq Ct^{(s-1)/2}.
\end{eqnarray*}
where $\bar{\bar{R}}_t(\bx,\by)=C_t\bar{\bar{R}}\left(\frac{\|\bx-\by\|^2}{4t}\right)$ and $\bar{\bar{R}}(r)=\int_r^\infty\bar{R}(s)\mathd s$.
\label{cor:upper_bound_int_R_t}
\end{lemma}

Now, we are ready to prove Theorem \ref{thm:regularity_boundary}.
\begin{proof} {\it of Theorem \ref{thm:regularity_boundary}}

The key point is to show that
  \begin{eqnarray}
\label{eq:est_boundary_whole}
    \left|\int_\M u(\bx)\left(\int_{\p\M}b_i(\by)\eta^j\bar{R}_t(\bx,\by)
\mathd \tau_\by-\bar{b}\right)\mathd \mu_\bx\right|\le C\sqrt{t} \;\|\mathbf{b}(\by)\|_{H^1(\M)} \|u\|_{H^1(\M)}
  \end{eqnarray}
Notice that
$$|\bar{b}|=\frac{1}{|\M|}\left|\int_\M \int_{\p\M}b_i(\by)\eta_i(\bx,\by)\rhk\mathd \tau_\by\mathd\bx\right|\le 
C\sqrt{t}\;\|\mathbf{b}(\by)\|_{L^2(\p\M)}\le C\sqrt{t}\;\|\mathbf{b}(\by)\|_{H^1(\M)}$$.
Then it is enough to show that
  \begin{eqnarray}
\label{eq:est_boundary}
    \left|\int_\M u(\bx)\left(\int_{\p\M}b_i(\by)\eta^j\bar{R}_t(\bx,\by)
\mathd \tau_\by\right)\mathd \mu_\bx\right|\le C\sqrt{t} \;\|\mathbf{b}(\by)\|_{H^1(\M)} \|u\|_{H^1(\M)}
  \end{eqnarray}
First, we have
\begin{eqnarray}
  &&|2t\nabla^j\rrhk+\eta^j\bar{R}_t(\bx,\by)|\le C|\xi|^2\rhk
\end{eqnarray}
where $\rrhk=C_t\bar{\bar{R}}\left(\frac{\|\bx-\by\|^2}{4t}\right)$ and $\bar{\bar{R}}(r)=\int_{r}^{\infty}\bar{R}(s)\mathd s$.

This would gives us that
\begin{eqnarray}
  &&\left|\int_\M u(\bx)\int_{\p\M}b_i(\by)\left(\eta^j\bar{R}_t(\bx,\by)+2t\nabla_i\rrhk\right)
\mathd \tau_\by\mathd \mu_\bx\right|\nonumber\\
&\le & C\int_\M |u(\bx)|\int_{\p\M}|b_i(\by)||\xi|^2\rhk\mathd \tau_\by\mathd \mu_\bx\nonumber\\
&\le &Ct \int_\M |u(\bx)|\int_{\p\M}|b_i(\by)|\rhk\mathd \tau_\by\mathd \mu_\bx\nonumber\\
&\le &Ct \int_{\p\M}|b_i(\by)| \int_\M |u(\bx)|\rhk\mathd \tau_\bx\mathd \mu_\by\nonumber\\
&\le &Ct\|\mathbf{b}(\by)\|_{L^2(\p \M)}\left(\int_{\p\M}\left(\int_\M\rhk\mathd\mu_\bx\right)
\left(\int_\M |u(\bx)|^2\rhk \mathd \mu_\bx\right)\mathd \tau_\by\right)^{1/2}\nonumber\\
&\le & Ct\|\mathbf{b}(\by)\|_{H^1(\M)}\left(\int_{\M} |u(\bx)|^2
\left(\int_{\p\M} \rhk \mathd \tau_\by\right)\mathd \mu_\bx\right)^{1/2}\nonumber\\
&\le & Ct^{3/4}\|\mathbf{b}(\by)\|_{H^1(\M)}\|u\|_{L^2(\M)}
\label{eq:est_boundary_1}
\end{eqnarray}
On the other hand, using Guass integral formula, we have
 \begin{eqnarray}
&&    \int_\M u(\bx)\int_{\p\M}b_i(\by)\nabla_i\rrhk\mathd \tau_\by\mathd \mu_\bx\nonumber\\
&=& \int_{\p\M} b_i(\by)\int_{\M}u(\bx)\nabla_i\rrhk\mathd \mu_\bx\mathd \tau_\by\nonumber\\
&=&\int_{\p\M} b_i(\by)\int_{\p\M}n_i(\bx)u(\bx)\rrhk\mathd \tau_\bx\mathd \tau_\by
-\int_{\p\M} \int_{\M}\text{div}_\bx[b_i(\by)u(\bx)]\rrhk\mathd \mu_\bx\mathd \tau_\by.
\label{eq:gauss_boundary}
  \end{eqnarray}
For the first term, we have
  \begin{eqnarray}
&&    \left|\int_{\p\M} b_i(\by)\int_{\p\M}n_i(\bx)u(\bx)\rrhk\mathd \tau_\bx\mathd \tau_\by\right|\nonumber\\
&\le &  \int_{\p\M}|b_i(\by)| \int_{\p\M}|u(\bx)|\rrhk\mathd \tau_\bx\mathd \tau_\by\nonumber\\
&\le &C\|\mathbf{b}(\by)\|_{L^2(\p \M)} \left(\int_{\p\M} \left(\int_{\p\M}|u(\bx)|\rrhk\mathd \tau_\bx\right)^2
\mathd \tau_\by\right)^{1/2}\nonumber\\
&\le &C\|\mathbf{b}(\by)\|_{H^1( \M)} \left(\int_{\p\M} \left(\int_{\p\M}\rrhk\mathd \tau_\bx\right)
 \left(\int_{\p\M}|u(\bx)|^2\rrhk\mathd \tau_\bx\right)
\mathd \tau_\by\right)^{1/2}\nonumber\\
&\le &Ct^{-1/4}\; \|\mathbf{b}(\by)\|_{H^1( \M)} \left(\int_{\p\M} 
 \left(\int_{\p\M}\rrhk\mathd \tau_\by\right)|u(\bx)|^2
\mathd \tau_\bx\right)^{1/2}\nonumber\\
&\le &Ct^{-1/2}\; \|\mathbf{b}(\by)\|_{H^1( \M)}\|u\|_{L^2(\p\M)}\nonumber\\
&\le &Ct^{-1/2}\; \|\mathbf{b}(\by)\|_{H^1( \M)}\|u\|_{H^1(\M)}.
\label{eq:est_boundary_2}
  \end{eqnarray}
We can also bound the second term of \eqref{eq:gauss_boundary} 
  \begin{eqnarray}
&&    \left|\int_{\p\M} \int_{\M}\text{div}_\bx[b_i(\by)u(\bx)]\rrhk\mathd \mu_\bx\mathd \tau_\by\right|\nonumber\\
&\le &C\|\mathbf{b}(\by)\|_{L^2(\p \M)} \int_{\p\M} \int_{\M}|\nabla u(\bx)|\rrhk\mathd \mu_\bx\mathd \tau_\by\nonumber\\
&\le &C \|\mathbf{b}(\by)\|_{H^1( \M)} \left(\int_{\p\M} \left(\int_\M \rrhk \mathd\mu_\bx\right)
\left(\int_{\M}|\nabla u(\bx)|^2\rrhk\mathd \mu_\bx\right)\mathd \tau_\by\right)^{1/2}\nonumber\\
&\le &C \|\mathbf{b}(\by)\|_{H^1( \M)} \left(\int_{\M} |\nabla u(\bx)|^2
\left(\int_{\p\M}\rrhk\mathd \tau_\by\right)\mathd \mu_\bx\right)^{1/2}\nonumber\\
&\le &C t^{-1/4}\;\|\mathbf{b}(\by)\|_{H^1( \M)} \|u\|_{H^1(\M)}.
\label{eq:est_boundary_3}
  \end{eqnarray}
Then, the inequality \eqref{eq:est_boundary} is obtained by \eqref{eq:est_boundary_1},
\eqref{eq:gauss_boundary}, \eqref{eq:est_boundary_2} and \eqref{eq:est_boundary_3}.

Now, using Lemma \ref{lem:elliptic_L_t}, we have
  \begin{eqnarray}
    \|u\|_{L^2(\M)}^2\le C \left<u,L_tu\right>=C\left|\int_\M u(\bx)\int_{\p\M}b_i(\by)\eta^j\bar{R}_t(\bx,\by)
\mathd \tau_\by\mathd \mu_\bx\right|\le C\sqrt{t}\;\|\mathbf{b}(\by)\|_{H^1( \M)} \|u\|_{H^1(\M)}.
\label{eq:est_l2_boundary}
  \end{eqnarray}
Denote $p(\bx)=\int_{\p\M}b_i(\by)\eta^j\bar{R}_t(\bx,\by)\mathd \tau_\by$, then direct calculation would give us that
\begin{eqnarray}
  \|p(\bx)\|_{L^2(\M)}&\le& Ct^{1/4}\|\mathbf{b}(\by)\|_{H^1( \M)},\\
  \|\nabla p(\bx)\|_{L^2(\M)}&\le& Ct^{-1/4}\|\mathbf{b}(\by)\|_{H^1( \M)}.
\end{eqnarray}
The integral equation $L_t u=p$ gives that
\begin{eqnarray}
u(\bx)=v(\bx)+\frac{t}{w_t(\bx)}\,p(\bx)
\end{eqnarray}
where
\begin{eqnarray}
  v(\bx)=\frac{1}{w_t(\bx)}\int_{\M}R_t(\bx,\by)u(\by)\mathd\mu_\by,\quad w_t(\bx)=\int_{\M}R_t(\bx,\by)\mathd\mu_\by
\end{eqnarray}
Then by Lemma \ref{lem:elliptic_v}, we have
\begin{eqnarray}
\|\nabla u\|_{L^2(\M)}^2&\le&   2\|\nabla v\|_{L^2(\M)}^2+ 2t^2\left\|\nabla \left(\frac{p(\bx)}{w_t(\bx)}\right)\right\|_{L^2(\M)}^2\nonumber\\
&\le & C \left<u,L_tu\right> + Ct\|p\|_{L^2(\M)}^2+Ct^2\|\nabla p\|_{L^2(\M)}^2\nonumber\\
&\le& C\sqrt{t}\;\|\mathbf{b}(\by)\|_{H^1( \M)} \|u\|_{H^1(\M)}+ Ct\|p\|_{L^2(\M)}^2+Ct^2\|\nabla p\|_{L^2(\M)}^2\nonumber\\
&\le& C\|\mathbf{b}(\by)\|_{H^1( \M)} \left(\sqrt{t}\|u\|_{H^1(\M)}+ Ct^{3/2}\right).
\label{eq:est_dl2_boundary}
\end{eqnarray}
Using \eqref{eq:est_l2_boundary} and \eqref{eq:est_dl2_boundary}, we have
\begin{eqnarray}
  \|u\|_{H^1(\M)}^2\le C\|\mathbf{b}(\by)\|_{H^1( \M)}\left(\sqrt{t}\|u\|_{H^1(\M)}+ Ct^{3/2}\right)
\end{eqnarray}
which proves the theorem.
\end{proof}


\subsection{Proof of Theorem \ref{thm:integral_error}}
\label{sec:integral_error}


\begin{proof}
Let $r(\bx)=L_t u-L_tu_t$, then we have
\begin{eqnarray}
  r(\bx)&=&-\frac{1}{t}\int_{\M}R_t(\bx,\by)(u(\bx)-u(\by))\mathd\mu_\by+2\int_{\p\M}\bar{R}_t(\bx,\by)g(\by)\mathd\tau_\by-\int_{\M}\bar{R}_t(\bx,\by)f(\by)\mathd\mu_\by
\nonumber\\
&=&-\frac{1}{t}\int_{\M}R_t(\bx,\by)(u(\bx)-u(\by))\mathd\mu_\by+\int_{\M}\bar{R}_t(\bx,\by)\Delta_\M u(\by)\mathd \mu_\by
\nonumber\\
&&+\frac{1}{t}\int_{\M}(\bx-\by)\cdot \nabla u(\by)R_t(\bx,\by)\mathd \mathd \mu_\by\nonumber\\
&=& -\frac{1}{t}\int_{\M}R_t(\bx,\by)(u(\bx)-u(\by)-(\bx-\by)\cdot \nabla u(\by))\mathd\mu_\by+\int_{\M}\bar{R}_t(\bx,\by)\Delta_\M u(\by)\mathd \mu_\by
\end{eqnarray}
Here we use that fact that $u$ is the solution of the Poisson equation with Newmann boundary condition \eqref{eqn:neumann}, such that
\begin{eqnarray}
  \int_{\p\M}\bar{R}_t(\bx,\by)g(\by)\mathd\tau_\by&=&\int_{\p\M}\bar{R}_t(\bx,\by)\frac{\p u}{\p \bn}(\by)\mathd\tau_\by\nonumber\\
&=&\int_{\M}\bar{R}_t(\bx,\by)\Delta_\M u(\by)\mathd \mu_\by+\int_{\M}\nabla_\by \bar{R}_t(\bx,\by)\nabla u(\by)\mathd \mu_\by\nonumber\\
&=&\int_{\M}\bar{R}_t(\bx,\by)\Delta_\M u(\by)\mathd \mu_\by+\frac{1}{t}\int_{\M}(\bx-\by)\cdot \nabla u(\by)R_t(\bx,\by)\mathd \mathd \mu_\by,\nonumber
\end{eqnarray}
and
\begin{eqnarray}
  \int_{\M}\bar{R}_t(\bx,\by)f(\by)\mathd\mu_\by=\int_{\M}\bar{R}_t(\bx,\by)\Delta_\M u(\by)\mathd \mu_\by.\nonumber
\end{eqnarray}

Denote
\begin{eqnarray}
  \mathcal{B}_{\bx}^r=\left\{\by\in \mathcal{M}: |\bx-\by|\le r\right\},\quad \mathcal{M}_{\bx}^t=\left\{\by\in \mathcal{M}: |\bx-\by|^2\le 32t\right\}
\end{eqnarray}
where $r>0$ is a positive number which is small enough.

Since the manifold $\mathcal{M}$ is compact, there exists $\bx_i\in \mathcal{M},\;i=1,\cdots,N$ such that
\begin{eqnarray}
  \mathcal{M}\subset \bigcup_{i=1}^N \mathcal{B}_{\bx_i}^r
\end{eqnarray}
By Proposition 6.1 in \cite{SS14}, we have there exist
a parametrization $\Phi_i: \Omega_i\subset\mathbb{R}^k \rightarrow U_i\subset \mathcal{M},; i=1,\cdots, N$, such that
\begin{itemize}
\item[1.] $\mathcal{B}_{\bx_i}^{2r}\subset U_i$ and $\Omega_i$ is convex.
\item[2.] $\Phi\in C^3(\Omega)$;
\item[3.] For any points $\bx, \by\in \Omega$, $\frac{1}{2}\left|\bx-\by\right| \leq \left\|\Phi_i(\bx)-\Phi_i(\by)\right\|  \leq 2\left|\bx-\by\right|$.
\end{itemize}

Moreover, we denote
$\Phi(\beta)=\bx, \Phi(\al)=\by, \xi=\beta-\al, \eta=\xi^i\p_i\Phi(\al)$,
and
\begin{eqnarray}
r_1(\bx)&=&-\frac{1}{t}\int_{\M}R_t(\bx,\by)\left(u(\bx)-u(\by)-(\bx-\by)\cdot \nabla u(\by)-\frac{1}{2}\eta^i\eta^j(\nabla^i\nabla^j u(\by))\right)\mathd\mu_\by\nonumber\\
r_2(\bx)&=&\frac{1}{2t}\int_{\M}R_t(\bx,\by)\eta^i\eta^j(\nabla^i\nabla^j u(\by))\mathd\mu_\by-\int_\M \eta^i(\nabla^i\nabla^j u(\by)\nabla^j\bar{R}_t(\bx,\by)\mathd \mu_\by\nonumber\\
r_3(\bx)&=& \int_\M \eta^i(\nabla^i\nabla^j u(\by)\nabla^j\bar{R}_t(\bx,\by)\mathd \mu_\by+\int_\M \mbox{div} \; \left(\eta^i(\nabla^i\nabla^j u(\by)\right)
\bar{R}_t(\bx,\by)\mathd \mu_\by\nonumber\\
r_4(\bx)&=&\int_\M \mbox{div} \; \left(\eta^i(\nabla^i\nabla^j u(\by)\right)\bar{R}_t(\bx,\by)\mathd \mu_\by+ \int_{\M}\bar{R}_t(\bx,\by)\Delta_\M u(\by)\mathd \mu_\by.\nonumber
\nonumber\\
\end{eqnarray}
then
\begin{eqnarray}
 r(\bx)=r_1(\bx)-r_1(\bx)-r_3(\bx)+r_4(\bx).
\end{eqnarray}
Next, we will prove the theorem by estimating above four terms one by one.

First, let us consider $r_1$. To simplify the notation, let
$$d(\bx,\by)=u(\bx)-u(\by)-(\bx-\by)\cdot \nabla u(\by)-\frac{1}{2}\eta^i\eta^j(\nabla^i\nabla^j u(\by)).$$
Then, we have
\begin{eqnarray}
  \int_\M |r_1(\bx)|^2\mathd\mu_\bx&=&  \int_\M \left|\int_\M R_t(\bx,\by)d(\bx,\by)\mathd \mu_\by\right|^2\mathd\mu_\bx\nonumber\\
&\le & \int_\M \left(\int_\M R_t(\bx,\by)\mathd \mu_\by\right)\left(\int_\M R_t(\bx,\by)|d(\bx,\by)|^2\mathd \mu_\by\right)\mathd\mu_\bx\nonumber\\
&\le & C\int_\M \int_\M R_t(\bx,\by)|d(\bx,\by)|^2\mathd \mu_\by\mathd\mu_\bx\nonumber\\
&\le& C\sum_{i=1}^N \int_\M\int_{B_{\bx_i}^r}R_t(\bx,\by)|d(\bx,\by)|^2\mathd \mu_\by\mathd\mu_\bx\nonumber\\
&=&C\sum_{i=1}^N \int_{B_{\bx_i}^{2r}}\int_{B_{\bx_i}^r}R_t(\bx,\by)|d(\bx,\by)|^2\mathd \mu_\by\mathd\mu_\bx\nonumber\\
&=&C\sum_{i=1}^N \int_{B_{\bx_i}^{r}}\left(\int_{\M_{\by}^t}R_t(\bx,\by)|d(\bx,\by)|^2\mathd \mu_\bx\right)\mathd\mu_\by.
\end{eqnarray}
Using the fact that $\Omega_i$ is convex and the Newton-Leibniz formula, we can get
\begin{eqnarray}
  d(\bx,\by)&=&u(\bx)-u(\by)-(\bx-\by)\cdot \nabla u(\by)-\frac{1}{2}\eta^i\eta^j(\nabla^i\nabla^j u(\by))\nonumber\\
&=&\xi^{i}\xi^{i'}\int_0^1\int_0^1\int_0^1s_1\frac{d}{d s_3}\left(\p_{i}\Phi^j(\al+s_3 s_1\xi)\p_{i'}\Phi^{j'}(\al+s_3s_2 s_1\xi)\nabla^{j'}\nabla^ju(\Phi(\al+s_3s_2s_1 \xi))\right)
\mathd s_3\mathd s_2\mathd s_1\nonumber\\
&=&\xi^{i}\xi^{i'}\xi^{i''}\int_0^1\int_0^1\int_0^1s_1^2s_2\p_{i}\Phi^j(\al+s_3 s_1\xi)\p_{i''}\p_{i'}\Phi^{j'}(\al+s_3s_2 s_1\xi)\nabla^{j'}\nabla^ju(\Phi(\al+s_3s_2s_1 \xi))
\mathd s_3\mathd s_2\mathd s_1\nonumber\\
&&+\xi^{i}\xi^{i'}\xi^{i''}\int_0^1\int_0^1\int_0^1s_1^2\p_{i''}\p_{i}\Phi^j(\al+s_3 s_1\xi)\p_{i'}\Phi^{j'}(\al+s_3s_2 s_1\xi)\nabla^{j'}\nabla^ju(\Phi(\al+s_3s_2s_1 \xi))
\mathd s_3\mathd s_2\mathd s_1\nonumber\\
&&+\xi^{i}\xi^{i'}\xi^{i''}\int_0^1\int_0^1\int_0^1s_1^2s_2\p_{i}\Phi^j(\al+s_3s_2 s_1\xi)\p_{i'}\Phi^{j'}(\al+s_3s_2 s_1\xi)\p_{i''}\Phi^{j''}(\al+s_3s_2 s_1\xi)\nonumber\\
&&\hspace{4cm}\nabla^{j''}\nabla^{j'}\nabla^ju(\Phi(\al+s_3s_2s_1 \xi))\mathd s_3\mathd s_2\mathd s_1\nonumber
\end{eqnarray}
Using this equality and $\Phi\in C^{3}(\Omega)$, it is easy to show that
\begin{eqnarray}
&&  \int_{B_{\bx_i}^{r}}\left(\int_{\M_{\by}^t}R_t(\bx,\by)|d(\bx,\by)|^2\mathd \mu_\bx\right)\mathd\mu_\by\nonumber\\
&\le & Ct^3 \int_0^1\int_0^1\int_0^1\int_{B_{\bx_i}^{r}}\int_{\M_{\by}^t}R_t(\bx,\by)\left|D^{2,3}u(\Phi_i(\al+s_3s_2s_1 \xi))\right|^2\mathd \mu_\bx\mathd\mu_\by
\mathd s_3\mathd s_2\mathd s_1\nonumber\\
&\le & Ct^3 \max_{0\le s\le 1}\int_{B_{\bx_i}^{r}}\int_{\M_{\by}^t}R_t(\bx,\by)\left|D^{2,3}u(\Phi_i(\al+s \xi))\right|^2\mathd \mu_\bx\mathd\mu_\by
\end{eqnarray}
where
\begin{eqnarray*}
  \left|D^{2,3}u(\bx)\right|^2=\sum_{j,j',j''=1}^n|\nabla^{j''}\nabla^{j'}\nabla^ju(\bx)|^2
+\sum_{j,j'=1}^n|\nabla^{j'}\nabla^ju(\bx)|^2.
\end{eqnarray*}

Let $\bz_i=\Phi_i(\al+s \xi),\; 0\le s\le 1$, then for any $\by\in B_{\bx_i}^r$ and $\bx\in \M_{\by}^t$,
\begin{eqnarray*}
  |\bz_i-\by|\le 2s|\xi|\le 4s|\bx-\by|\le 8s\sqrt{t},\quad |\bz_i-\bx_i|\le |\bz_i-\by|+|\by-\bx_i|\le r+8s\sqrt{t}.
\end{eqnarray*}
We can assume that $t$ is small enough such that $8\sqrt{t}\le r$, then we have
\begin{eqnarray*}
  \bz_i\in B_{\bx_i}^{2r}.
\end{eqnarray*}
After changing of variable, we obtain
\begin{eqnarray}
&&  \int_{B_{\bx_i}^{r}}\int_{\M_{\by}^t}R_t(\bx,\by)\left|D^{2,3}u(\Phi_i(\al+s_3s_2s_1 \xi))\right|^2\mathd \mu_\bx\mathd\mu_\by\nonumber\\
&\le & \frac{C}{\delta_0} \int_{B_{\bx_i}^{r}}\int_{B_{\bx_i}^{2r}}\frac{1}{s^k}R\left(\frac{|\bz_i-\by|^2}{128s^2t}\right)
\left|D^{2,3}u(\bz_i)\right|^2\mathd \mu_{\bz_i}\mathd\mu_\by\nonumber\\
&=&\frac{C}{\delta_0} \int_{B_{\bx_i}^{r}}\frac{1}{s^k}R\left(\frac{|\bz_i-\by|^2}{128s^2t}\right)\mathd\mu_\by
\int_{B_{\bx_i}^{2r}}\left|D^{2,3}u(\bz_i)\right|^2\mathd \mu_{\bz_i}\nonumber\\
&\le & C \int_{B_{\bx_i}^{2r}}\left|D^{2,3}u(\bx)\right|^2\mathd \mu_{\bx}
\end{eqnarray}
This estimate would give us that
\begin{eqnarray}
\label{est:r1}
  \|r_1(\bx)\|_{L^2(\M)}\le Ct^{1/2}\|u\|_{H^3(\M)}
\end{eqnarray}
Now, we turn to estimate the gradient of $r_1$.
\begin{eqnarray*}
  \int_\M |\nabla_\bx r_1(\bx)|^2\mathd\mu_\bx&\le &  C\int_\M \left|\int_\M \nabla_\bx R_t(\bx,\by)d(\bx,\by)\mathd \mu_\by\right|^2\mathd\mu_\bx
+C\int_\M \left|\int_\M  R_t(\bx,\by)\nabla_\bx d(\bx,\by)\mathd \mu_\by\right|^2\mathd\mu_\bx.
\end{eqnarray*}
Using the same techniques in the calculation of $\|r_1(\bx)\|_{L^2(\M)}$, we can get that
the first term of right hand side can bounded as follows
\begin{eqnarray}
  \int_\M \left|\int_\M \nabla_\bx R_t(\bx,\by)d(\bx,\by)\mathd \mu_\by\right|^2\mathd\mu_\bx\le C  \|u\|_{H^3(\M)}^2.
\end{eqnarray}
The estimation of second term is a little involved. First, we have
\begin{eqnarray}
&&\int_\M \left|\int_\M  R_t(\bx,\by)\nabla_\bx d(\bx,\by)\mathd \mu_\by\right|^2\mathd\mu_\bx\nonumber\\
&\le & \int_\M \left(\int_\M R_t(\bx,\by)\mathd \mu_\by\right)\left(\int_\M R_t(\bx,\by)|\nabla_\bx d(\bx,\by)|^2\mathd \mu_\by\right)\mathd\mu_\bx\nonumber\\
&\le & C\int_\M \int_\M R_t(\bx,\by)|\nabla_\bx d(\bx,\by)|^2\mathd \mu_\by\mathd\mu_\bx\nonumber\\
&\le& C\sum_{i=1}^N \int_\M\int_{B_{\bx_i}^r}R_t(\bx,\by)|\nabla_\bx d(\bx,\by)|^2\mathd \mu_\by\mathd\mu_\bx\nonumber\\
&=&C\sum_{i=1}^N \int_{B_{\bx_i}^{2r}}\int_{B_{\bx_i}^r}R_t(\bx,\by)|\nabla_\bx d(\bx,\by)|^2\mathd \mu_\by\mathd\mu_\bx\nonumber\\
&=&C\sum_{i=1}^N \int_{B_{\bx_i}^{r}}\left(\int_{\M_{\by}^t}R_t(\bx,\by)|\nabla_\bx d(\bx,\by)|^2\mathd \mu_\bx\right)\mathd\mu_\by.
\end{eqnarray}
Using Newton-Leibniz formula, we have
\begin{eqnarray}
  d(\bx,\by)
&=&\xi^{i}\xi^{i'}\int_0^1\int_0^1s_1\left(\p_{i}\Phi^j(\al+ s_1\xi)\p_{i'}\Phi^{j'}(\al+s_2 s_1\xi)\nabla^{j'}\nabla^ju(\Phi(\al+s_2s_1 \xi))\right)
\mathd s_2\mathd s_1\nonumber\\
&&-\xi^{i}\xi^{i'}\int_0^1\int_0^1s_1\left(\p_{i}\Phi^j(\al)\p_{i'}\Phi^{j'}(\al)\nabla^{j'}\nabla^ju(\Phi(\al))\right)
\mathd s_2\mathd s_1
\end{eqnarray}
Then the gradient of $d(\bx,\by)$ has following representation,
\begin{eqnarray}
\nabla_\bx  d(\bx,\by)
&=&\xi^{i}\xi^{i'}\nabla_\bx\left(\int_0^1\int_0^1s_1\left(\p_{i}\Phi^j(\al+ s_1\xi)\p_{i'}\Phi^{j'}(\al+s_2 s_1\xi)\nabla^{j'}\nabla^ju(\Phi(\al+s_2s_1 \xi))\right)
\mathd s_2\mathd s_1\right)\nonumber\\
&&\hspace{-20mm}+\nabla_\bx\left(\xi^{i}\xi^{i'}\right)\int_0^1\int_0^1\int_0^1s_1\frac{d}{d s_3}\left(\p_{i}\Phi^j(\al+ s_3s_1\xi)\p_{i'}\Phi^{j'}(\al+s_3s_2 s_1\xi)
\nabla^{j'}\nabla^ju(\Phi(\al+s_3s_2s_1 \xi))\right)
\mathd s_3\mathd s_2\mathd s_1\nonumber\\
&=&d_1(\bx,\by)+d_2(\bx,\by).
\end{eqnarray}
For $d_1$, we have
\begin{eqnarray}
&&  \int_{B_{\bx_i}^{r}}\left(\int_{\M_{\by}^t}R_t(\bx,\by)|d_1(\bx,\by)|^2\mathd \mu_\bx\right)\mathd\mu_\by\nonumber\\
&\le& Ct^2\int_0^1\int_0^1 \int_{B_{\bx_i}^{r}}\left(\int_{\M_{\by}^t}R_t(\bx,\by)|D^{2,3}u(\Phi(\al+s_2s_1 \xi))|^2\mathd \mu_\bx\right)\mathd\mu_\by \mathd s_2\mathd s_1\nonumber\\
&\le& Ct^2\max_{0\le s\le 1} \int_{B_{\bx_i}^{r}}\left(\int_{\M_{\by}^t}R_t(\bx,\by)|D^{2,3}u(\Phi(\al+s \xi))|^2\mathd \mu_\bx\right)\mathd\mu_\by
\end{eqnarray}
which means that
\begin{eqnarray}
\label{eqn:est-dr1-d1}
   \int_{B_{\bx_i}^{r}}\left(\int_{\M_{\by}^t}R_t(\bx,\by)|d_1(\bx,\by)|^2\mathd \mu_\bx\right)\mathd\mu_\by\le C  \int_{B_{\bx_i}^{2r}}|D^{2,3}u(\bx)|^2\mathd \mu_\bx
\end{eqnarray}
For $d_2$, we have
\begin{eqnarray}
&&  d_2(\bx,\by)\nonumber\\
  &=&\nabla_\bx\left(\xi^{i}\xi^{i'}\right)\int_{[0,1]^3}s_1\frac{d}{d s_3}\left(\p_{i}\Phi^j(\al+ s_3s_1\xi)\p_{i'}\Phi^{j'}(\al+s_3s_2 s_1\xi)
\nabla^{j'}\nabla^ju(\Phi(\al+s_3s_2s_1 \xi))\right)
\mathd s_3\mathd s_2\mathd s_1\nonumber\\
&=&\nabla_\bx\left(\xi^{i}\xi^{i'}\right)\xi^{i''}\int_{[0,1]^3}s_1^2s_2\p_{i}\Phi^j(\al+s_3 s_1\xi)\p_{i''}\p_{i'}\Phi^{j'}(\al+s_3s_2 s_1\xi)\nabla^{j'}\nabla^ju(\Phi(\al+s_3s_2s_1 \xi))
\mathd s_3\mathd s_2\mathd s_1\nonumber\\
&&+\nabla_\bx\left(\xi^{i}\xi^{i'}\right)\xi^{i''}\int_{[0,1]^3}s_1^2\p_{i''}\p_{i}\Phi^j(\al+s_3 s_1\xi)\p_{i'}\Phi^{j'}(\al+s_3s_2 s_1\xi)\nabla^{j'}\nabla^ju(\Phi(\al+s_3s_2s_1 \xi))
\mathd s_3\mathd s_2\mathd s_1\nonumber\\
&&+\nabla_\bx\left(\xi^{i}\xi^{i'}\right)\xi^{i''}\int_{[0,1]^3}s_1^2s_2\p_{i}\Phi^j(\al+s_2 s_1\xi)\p_{i'}\Phi^{j'}(\al+s_3s_2 s_1\xi)\p_{i''}\Phi^{j''}(\al+s_3s_2 s_1\xi)\nonumber\\
&&\hspace{4cm}\nabla^{j''}\nabla^{j'}\nabla^ju(\Phi(\al+s_3s_2s_1 \xi))\mathd s_3\mathd s_2\mathd s_1\nonumber
\end{eqnarray}
This formula tells us that
\begin{eqnarray}
&&  \int_{B_{\bx_i}^{r}}\left(\int_{\M_{\by}^t}R_t(\bx,\by)|d_2(\bx,\by)|^2\mathd \mu_\bx\right)\mathd\mu_\by\nonumber\\
&\le& Ct^2\int_0^1\int_0^1\int_0^1 \int_{B_{\bx_i}^{r}}\left(\int_{\M_{\by}^t}R_t(\bx,\by)|D^{2,3}u(\Phi(\al+s_2s_1 \xi))|^2\mathd \mu_\bx\right)\mathd\mu_\by \mathd s_3\mathd s_2\mathd s_1\nonumber\\
&\le& Ct^2\max_{0\le s\le 1} \int_{B_{\bx_i}^{r}}\left(\int_{\M_{\by}^t}R_t(\bx,\by)|D^{2,3}u(\Phi(\al+s \xi))|^2\mathd \mu_\bx\right)\mathd\mu_\by
\end{eqnarray}
Using the same arguments as that in the calculation of $\|r_1\|_{L^2(\M)}$, we have
\begin{eqnarray}
\label{eqn:est-dr1-d2}
   \int_{B_{\bx_i}^{r}}\left(\int_{\M_{\by}^t}R_t(\bx,\by)|d_2(\bx,\by)|^2\mathd \mu_\bx\right)\mathd\mu_\by\le C  \int_{B_{\bx_i}^{2r}}|D^3u(\bx)|^2\mathd \mu_\bx
\end{eqnarray}
Combining \eqref{eqn:est-dr1-d1} and \eqref{eqn:est-dr1-d2}, we have
\begin{eqnarray}
\label{est:dr1}
  \|\nabla r_1(\bx)\|_{L^2(\M)}\le C\|u\|_{H^3(\M)}
\end{eqnarray}
The estimates of $r_2$, $r_3$ and $r_4$ are similar as those in our previous paper \cite{SS14}. In order to
make this proof self-consistent, we also give a complete proof of this part.

For $r_2$, first, notice that
\begin{eqnarray}
  \nabla^j\bar{R}_t(\bx,\by)&=&\frac{1}{2t}\p_{m'}\Phi^j(\al) g^{m'n'}\p_{n'}\Phi^i(\al) (x^i-y^i)\hk,\nonumber\\
\frac{\eta^j}{2t}R_t(\bx,\by)&=&\frac{1}{2t}\p_{m'}\Phi^j(\al) g^{m'n'}\p_{n'}\Phi^i(\al) \xi^{i'}\p_{i'}\Phi^i\hk.\nonumber
\end{eqnarray}
Then, we have
\begin{eqnarray}
  &&\nabla^j\bar{R}_t(\bx,\by)-\frac{\eta^j}{2t}R_t(\bx,\by)\nonumber\\
&=&\frac{1}{2t}\p_{m'}\Phi^i g^{m'n'}\p_{n'}\Phi^j (x^j-y^j)-\eta^j\hk\nonumber\\
&=&\frac{1}{2t}\p_{m'}\Phi^i g^{m'n'}\p_{n'}\Phi^j \left(x^j-y^j-\xi^{i'}\p_{i'}\Phi^j\right)\hk\nonumber\\
&=&\frac{1}{2t}\xi^{i'}\xi^{j'}\p_{m'}\Phi^i g^{m'n'}\p_{n'}\Phi^j \left(\int_0^1\int_0^1s\p_{j'}\p_{i'}\Phi^j(\al+\tau s \xi)\mathd \tau\mathd s\right)\hk\nonumber
\end{eqnarray}
Thus, we get
\begin{eqnarray}
 \left|\nabla^j\bar{R}_t(\bx,\by)-\frac{\eta^j}{2t}R_t(\bx,\by)\right|\le \frac{C|\xi|^2}{t}\hk\nonumber\\
\left|\nabla_\bx\left(\nabla^j\bar{R}_t(\bx,\by)-\frac{\eta^j}{2t}R_t(\bx,\by)\right)\right|\le \frac{C|\xi|}{t}\hk+\frac{C|\xi|^3}{t^2}|R'_t(\bx,\by)|\nonumber
\end{eqnarray}
Then, we have following bound for $r_2$,
\begin{eqnarray}
\label{est:r2}
\int_{\M}|r_2(\bx)|^2\mathd \mu_\bx&\le& Ct\int_\M  \left(\int_\M \hk |D^2u(\by)|\mathd \mu_\by\right)^2\mathd\mu_\bx \nonumber\\
&\le &  Ct\int_\M  \left(\int_\M \hk \mathd\mu_\by\right)\int_\M \hk |D^2u(\by)|^2\mathd \mu_\by\mathd\mu_\bx \nonumber\\
&\le& Ct  \max_{\by}\left(\int_\M \hk\mathd \mu_\bx\right) \int_\M|D^2u(\by)|^2\mathd \mu_\by \nonumber\\
&\le& Ct\|u\|_{H^2(\M)}^2.
\end{eqnarray}
Similarly, we have
\begin{eqnarray}
 \int_M |\nabla r_2(\bx)|^2\mathd\mu_\bx
&\le &  Ct\int_\M  \left(\int_\M\nabla_\bx \hk \mathd\mu_\by\right)\int_\M \nabla_\bx\hk
|D^2u(\by)|^2\mathd \mu_\by\mathd\mu_\bx \nonumber\\
 &\le& C\sqrt{t}  \max_{\by}\left(\int_\M \nabla_\bx\hk\mathd \mu_\bx\right)
\max_{\bx}\left(\int_\M \nabla_\bx\hk\mathd \mu_\by\right) \int_\M|D^2u(\by)|^2\mathd \mu_\by\nonumber\\
&\le&  C\|u\|_{H^2(\M)}^2.
\label{est:dr2}
\end{eqnarray}
$r_3$ is relatively easy to estimate by using the well known Gauss formula.
\begin{eqnarray}
\label{est:r3}
  r_3(\bx)=\int_{\p\mathcal{M}}n^{j}\eta^{i}(\nabla^{i} \nabla^{j}u) \rhk \mathd\tau_\by
\end{eqnarray}
Now, we turn to bound the last term $r_4$. Notice that
\begin{eqnarray}
  \nabla^j\left(\nabla^{j}u\right)&=&(\p_{k'} \Phi^j)g^{k'l'}\p_{l'}\left((\p_{m'}\Phi^j)g^{m'n'}(\p_{n'} u )\right)\nonumber\\
&=&(\p_{k'}\Phi^j)g^{k'l'}\left(\p_{l'}(\p_{m'}\Phi^j)\right)g^{m'n'}(\p_{n'}u)+(\p_{k'}\Phi^j)g^{k'l'}(\p_{m'}\Phi^j)\p_{l'}\left(g^{m'n'}(\p_{n'}u)\right)\nonumber\\
&=&\frac{1}{\sqrt{g}}(\p_{m'}\sqrt{g}) g^{m'n'}(\p_{n'}u)+\p_{m'}\left(g^{m'n'}(\p_{n'}u)\right)\nonumber\\
&=&\frac{1}{\sqrt{g}}\p_{m'}\left(\sqrt{g} g^{m'n'}(\p_{n'}u)\right) \nonumber\\
&=&\Delta_\M u.
\label{eqn:grad2Delta}
\end{eqnarray}
Here we use the fact that
\begin{eqnarray}
  (\p_{k'}\Phi^j)g^{k'l'}\left(\p_{l'}(\p_{m'}\Phi^j)\right)&=&(\p_{k'}\Phi^j)g^{k'l'}\left(\p_{m'}(\p_{l'}\Phi^j)\right)\nonumber\\
&=&(\p_{m'}(\p_{k'}\Phi^j))g^{k'l'}(\p_{l'}\Phi^j)\nonumber\\
&=&\frac{1}{2}g^{k'l'}\p_{m'}(g_{k'l'})\nonumber \\
&=&\frac{1}{\sqrt{g}}(\p_{m'}\sqrt{g})
\end{eqnarray}
Moreover, we have
\begin{eqnarray}
&& g^{i'j'}(\p_{j'}\Phi^{j})(\p_{i'}\xi^{l})(\p_{l}\Phi^{i})(\nabla^{i}\nabla^{j}u)\nonumber\\
&=&- g^{i'j'}(\p_{j'}\Phi^{j})(\p_{i'}\Phi^{i})(\nabla^{i}\nabla^{j}u)\nonumber\\
&=&- g^{i'j'}(\p_{j'}\Phi^{j})(\p_{i'}\Phi^{i})(\p_{m'}\Phi^{i})g^{m'n'}\p_{n'}\left(\nabla^{j}u \right)\nonumber\\
&=&- g^{i'j'}(\p_{j'}\Phi^{j})\p_{i'}\left(\nabla^{j}u\right)\nonumber\\
&=&- \nabla^{j}\left(\nabla^{j}u\right) .
\label{eqn:div}
\end{eqnarray}
where the first equalities are due to that
$\p_{i'}\xi^l = -\delta_{i'}^l$.
Then we have
\begin{eqnarray}
  &&\mbox{div} \; \left(\eta^i(\nabla^i\nabla^j u(\by))\right)+\Delta_\M u\nonumber\\
&=&\frac{1}{\sqrt{|g|}}\p_{i'}\left(\sqrt{|g|}g^{i'j'}(\p_{j'}\Phi^j) \xi^{l}(\p_l \Phi^i)(\nabla^i\nabla^j u(\by))\right)
-g^{i'j'}(\p_{j'}\Phi^{j})(\p_{i'}\xi^{l})(\p_{l}\Phi^{i})(\nabla^{i}\nabla^{j}u)\nonumber\\
&=&\frac{\xi^l}{\sqrt{|g|}}\p_{i'}\left(\sqrt{|g|}g^{i'j'}(\p_{j'}\Phi^j) (\p_l \Phi^i)(\nabla^i\nabla^j u(\by))\right)
\end{eqnarray}
Here we use the equalities \eqref{eqn:grad2Delta}, \eqref{eqn:div}, $\eta^i = \xi^{l} \p_{i'}\Phi^l$ and the definition of $\mbox{div}$,
\begin{equation}
\mbox{div}X = \frac{1}{\sqrt{|g|}}\p_{i'} (\sqrt{|g|}g^{i'j'}\p_{j'}\Phi^k X^k).
\label{eqn:divergence}
\end{equation}
where $X$ is a smooth tangent vector
field on $\M$ and $(X^1,\dots,X^d)^t$ is its representation in embedding coordinates.

Hence,
\begin{eqnarray}
  r_4(\bx)=\int_\mathcal{M}\frac{\xi^l}{\sqrt{|g|}}\p_{i'}\left(\sqrt{|g|}g^{i'j'}(\p_{j'}\Phi^j) (\p_l \Phi^i)(\nabla^i\nabla^j u(\by))\right) \rhk \mathd \mu_\by\nonumber
\end{eqnarray}
Then it is easy to get that
\begin{eqnarray}
  \label{est:r4}
  \|r_4(\bx)\|_{L^2(\M)}&\le& C t^{1/2}\|u\|_{H^3(\M)},\\
\|\nabla r_4(\bx)\|_{L^2(\M)}&\le & C\|u\|_{H^3(\M)}.
\label{est:dr4}
\end{eqnarray}
The proof is complete by combining \eqref{est:r1},\eqref{est:dr1},\eqref{est:r2},\eqref{est:dr2},\eqref{est:r3},\eqref{est:r4},\eqref{est:dr4}.
\end{proof}


\section{Proof of Theorem \ref{thm:converge_c1}}
\label{sec:converge-c1}
First, we need the following theorem which have been proved in \cite{SS14}.
\begin{theorem}
Assume $\M$ and $\p\M$ are $C^\infty$ and the input data $(P,\bV)$ is an $h$-integral approximation of $\M$.
Let
$f\in C(\M)$
in both problems, then there exists constants
$C, T_0$ depending only on $\M$ and  $\p \M$ so that
\begin{eqnarray}
\|T_{t,h}f-T_t f\|_{H^1(\M)} &\leq& \frac{Ch}{t^{3/2}}\|f\|_{\infty},
\end{eqnarray}
as long as $t\le T_0$ and $\frac{h}{\sqrt{t}}\le T_0$.
\label{thm:dis_error}
\end{theorem}

Then the main idea to prove Theorem \ref{thm:converge_c1} is to lift the covergence from $H^1$ to $C^1$ by using the fact
that $T_tu$ and $T_{t,h}u$ have higher order regularity for any $u\in C(\M)$. The details are given as following.
\begin{proof} {\it of Theorem \ref{thm:converge_c1}}:

First, for any $u\in C^1(\M)$, we know that $T_tu$ and $T_{t,h}u$ have following representations
  \begin{eqnarray}
    T_tu&=&\frac{1}{w_t(\bx)}\int_{\M}R_t(\bx,\by)T_tu(\by)\mathd\by+\frac{t}{w_t(\bx)}\int_\M\bar{R}(\bx,\by)u(\by)\mathd\by,\\
T_{t,h}u&=&\frac{1}{w_{t,h}(\bx)}\sum_{i}R_t(\bx,\bx_i)T_{t,h}u(\bx_i)V_i+\frac{t}{w_{t,h}(\bx)}\sum_{i}\bar{R}(\bx,\bx_i)u(\bx_i)V_i.
  \end{eqnarray}
Denote
  \begin{eqnarray}
    T^1_tu&=&\frac{1}{w_{t,h}(\bx)}\int_{\M}R_t(\bx,\by)T_tu(\by)\mathd\by+\frac{t}{w_{t,h}(\bx)}\int_\M\bar{R}(\bx,\by)u(\by)\mathd\by,\\
    T^2_tu&=&\frac{1}{w_{t,h}(\bx)}\int_{\M}R_t(\bx,\by)T_{t,h}u(\by)\mathd\by+\frac{t}{w_{t,h}(\bx)}\int_\M\bar{R}(\bx,\by)u(\by)\mathd\by,\\
  \end{eqnarray}
Direct calculation would give that
  \begin{eqnarray}
    \|w_t(\bx)-w_{t,h}(\bx)\|_\infty\le \frac{Ch}{t^{1/2}},\quad \|\nabla w_t(\bx)-\nabla w_{t,h}(\bx)\|_\infty\le \frac{Ch}{t}
  \end{eqnarray}
and
  \begin{eqnarray}
    \left|\int_{\M}R_t(\bx,\by)T_tu(\by)\mathd\by\right|&\le& C t^{-k/4}\|T_tu\|_{L^2}\le Ct^{-k/4}\|u\|_{H^1}\le Ct^{-k/4}\|u\|_{C^1},\nonumber\\
\left|\nabla_\bx\int_{\M}R_t(\bx,\by)T_tu(\by)\mathd\by\right|&\le& C t^{-(k+2)/4}\|T_tu\|_{L^2}\le Ct^{-(k+2)/4}\|u\|_{C^1}
  \end{eqnarray}
and
  \begin{eqnarray}
    \left|\int_\M\bar{R}(\bx,\by)u(\by)\mathd\by\right|\le C\|u\|_{\infty},\quad  \left|\nabla_\bx \int_\M\bar{R}(\bx,\by)u(\by)\mathd\by\right|
\le Ct^{-1/2}\|u\|_{\infty}
  \end{eqnarray}
Using above inequalites, we have
  \begin{eqnarray}
    \left|T_tu - T^1_tu\right|&\le&  \frac{Ch}{t^{(k+2)/4}}\|u\|_\infty,\\
    \left|\nabla (T_tu - T^1_tu)\right|&\le&  \frac{Ch}{t^{k/4+1}}\|u\|_\infty,
  \end{eqnarray}
which proves that
  \begin{eqnarray}
\label{eqn:est-t1}
        \left\|T_tu - T^1_tu\right\|_{C^1}&\le&  \frac{Ch}{t^{k/4+1}}\|u\|_\infty.
  \end{eqnarray}
Secondly, we have
  \begin{eqnarray}
    \left| T^1_tu- T^2_tu\right|&=&\left|\frac{1}{w_{t,h}(\bx)}\int_{\M}R_t(\bx,\by)\left(T_tu(\by)-T_{t,h}u(\by)\right)\mathd\by\right|\nonumber\\
&\le & Ct^{-k/4}\left\|T_tu-T_{t,h}u\right\|_{L^2}\le \frac{Ch}{t^{k/4+3/2}}\|u\|_\infty.
  \end{eqnarray}
and
 \begin{eqnarray}
    \left| \nabla \left(T^1_tu- T^2_tu\right)\right|&=&\left|\nabla_\bx \left(\frac{1}{w_{t,h}(\bx)}\int_{\M}R_t(\bx,\by)\left(T_tu(\by)-T_{t,h}u(\by)\right)\mathd\by\right)\right|\nonumber\\
&\le & Ct^{-k/4+1/2}\left\|T_tu-T_{t,h}u\right\|_{L^2}\le \frac{Ch}{t^{k/4+2}}\|u\|_\infty.
  \end{eqnarray}
This implies that
  \begin{eqnarray}
\label{eqn:est-t12}
     \left\|T^1_tu - T^2_tu\right\|_{C^1}&\le&  \frac{Ch}{t^{k/4+2}}\|u\|_\infty
  \end{eqnarray}
Using Theorem \ref{thm:bound}, we have
  \begin{eqnarray}
    \left|T_{t,h}u\right|&\le&  Ct^{-k/4}\|u\|_{\infty},\\
\left|\nabla T_{t,h}u\right|&\le&  Ct^{-(k+2)/4}\|u\|_{\infty}.
  \end{eqnarray}
Thus,
  \begin{eqnarray}
    \left|T_{t,h}u-T^2_tu\right|&\le& \frac{Ch}{t^{(k+2)/4}}\|u\|_{\infty},
 \\
    \left|\nabla\left(T_{t,h}u-T^2_tu\right)\right|&\le& \frac{Ch}{t^{k/4+1}}\|u\|_{\infty}.
  \end{eqnarray}
which also reads
  \begin{eqnarray}
\label{eqn:est-t2}
    \left\|T_{t,h}u-T^2_tu\right\|_{C^1}\le \frac{Ch}{t^{k/4+1}}\|u\|_{\infty}
  \end{eqnarray}
The proof is complete by combining \eqref{eqn:est-t1}, \eqref{eqn:est-t12} and \eqref{eqn:est-t2}.
\end{proof}

\section{Proof of Theorem \ref{thm:bound}}
\label{sec:bound}
In this section, we prove the bounds on the solution operators $T_t$ and $T_{t, h}$.
First, we need following theorem whose proof can be found in \cite{SS14}.
\begin{theorem}
If both the submanifolds $\M$ and $\p\M$ are $C^\infty$ smooth and the input data $(P, S, \bV, \bA)$ is an $h$-integral
approximation of $(\M, \p\M)$,  then there exist constants $C>0,\, C_0>0$ independent on $t$ so that for 
any ${\bf u} = (u_1, \cdots, u_{|P|})^t$ with $\sum_{i=1}^{|P|} u_iV_i = 0$ and for any sufficient small $t$ 
and $\frac{h}{\sqrt{t}}$ 
\begin{equation}
\left<{\bf u}, \mathcal{L} {\bf u} \right>_{\bV} \geq C(1-\frac{C_0h}{\sqrt{t}}) \left<{\bf u}, {\bf u}\right>_{\bV},\nonumber
\end{equation}
where $\left<{\bf u}, {\bf v}\right>_{\bV} = \sum_{i=1}^{|P|} u_iv_iV_i$ for any ${\bf u} = (u_1, \cdots, u_{|P|})^t, 
{\bf v} =(v_1, \cdots, v_{|P|})^t$. 
\label{thm:elliptic_L}
\end{theorem}

Now, we have enough materials to prove Theorem~\ref{thm:bound}.
\begin{proof} {\it of Theorem~\ref{thm:bound}}.\\
Let $u = T_t(f)$. From Lemma \ref{lem:elliptic_L_t}, there exists a constant $C$ independent of $t$
so that 
  \begin{eqnarray*}
\label{eq:est_u_l2}
    \|u\|^2_{L^2(\M)} &\le&  C C_t\int_{\M}\left(\int_{\M}R\left(\frac{|\bx-\by|^2}{4t}\right)f(\by)\mathd\mu_\by \right)u(\bx) \mathd\mu_\bx\\ 
	 &\le& C\|u\|_{L^2(\mathcal{M})}\|f\|_{L^2(\mathcal{M})}, 
  \end{eqnarray*}
which means $\|u\|_{L^2(\M)}\leq C\|f\|_{L^2(\mathcal{M})}.$ At the same time, we can write 
\begin{eqnarray}
  u(\bx)=\frac{1}{w_t(\bx)}\int_{\mathcal{M}}R_t(\bx,\by)u(\by)\mathd \mu_\by-\frac{t}{w_t(\bx)}\int_{\mathcal{M}}\bar{R}_t(\bx,\by)f(\by)\mathd \mu_\by. \nonumber
\end{eqnarray}
Since the kernel function, $R$, is chosen to be $C^1$, $u$ is also belong to $C^1$ and 
\begin{eqnarray}
  |u(\bx)|^2 &\leq& C\left(\int_{\mathcal{M}}R_t(\bx,\by)\mathd \mu_\by\right)\left(\int_{\mathcal{M}}R_t(\bx,\by)u^2(\by)\mathd \mu_\by\right) \nonumber\\
&&+Ct^2\left(\int_{\mathcal{M}}\bar{R}_t(\bx,\by)\mathd \mu_\by\right)\left(\int_{\mathcal{M}}\bar{R}_t(\bx,\by)f^2(\by)\mathd \mu_\by\right)   \nonumber \\
  				  &\leq& C C_t \left(\|u\|^2_{L^2(\M)} + t^2\|f\|^2_{L^2(\M)}\right) \leq  CC_t\|f\|^2_{L^2(\M)} \nonumber
\end{eqnarray}
Similarly, we have 
$$\|\nabla u\|_{\infty} \leq \frac{C}{t^{1/2}}(\|u\|_\infty+t\|f\|_\infty) \leq  \frac{CC_t^{1/2}}{t^{1/2}}\|f\|_\infty.$$

Now, we turn to bound $T_{t,h}$. We can write
\begin{eqnarray}
  T_{t,h}(f)(\bx)&=&\frac{1}{w_{t,h}(\bx)}\sum_{\bfp_j\in P}R_t(\bx,\bfp_j)u_jV_j-\frac{t}{w_{t,h}(\bx)}\sum_{\bfp_j\in P}\bar{R}_t(\bx,\bfp_j)f(\bfp_j)V_j,\nonumber
\end{eqnarray}
where $\bfu = (u_1, \cdots, u_{|P|})^t$ with $\sum_{i=1}^{|P|} u_i V_i =0$ solves the problem \eqref{eqn:dis}. 

From Theorem~\ref{thm:elliptic_L}, we can easily get that $\left(\sum_{j=1}^{|P|}u_j^2V_j\right)^{1/2} \leq C\|f\|_\infty$ for some constant $C$ independent on $t$. It is obvious
that for some constant $C$ independent on $t$,  
\begin{equation}
|T_{t,h}(f)(\bx)| \leq C C_t^{1/2} \left(\sum_{j=1}^{|P|}u_j^2V_j\right)^{1/2} + Ct\|f\|_\infty 
\leq CC_t^{1/2} \|f\|_\infty,\nonumber
\end{equation} 
and
\begin{equation}
|\nabla T_{t,h}(f)(\bx)| \leq \frac{C C_t^{1/2}}{t^{1/2}} \left(\sum_{j=1}^{|P|}u_j^2V_j\right)^{1/2} + Ct^{1/2}\|f\|_\infty 
\leq \frac{C C_t^{1/2}}{t^{1/2}}\|f\|_\infty.\nonumber
\end{equation}
\end{proof}


\section{Conclusion and Future Work}
\label{sec:conclusion}

In this paper, we proved that the convergence of the point integral method for the spectra of the Laplace-Beltrami operator on
point cloud. 
And the rate of convergence is also obtained.
This work builds
a solid mathematical foundation for many Laplace spectra based algorithm.

In many applications, the sample points of the manifold are draw according to some probability distribution.
Then one interesting problem is to study the performance of the point integral method on the random samples
as the number of sample points tends to infinity. 
Based on the results reported in this paper
We can show that with overwhelming probability, the spectra given by the point integral method
converge to the spectra of following eigen problem
\begin{equation}
\left\{\begin{array}{rl}
      -\frac{1}{p^2(\bx)}\nabla\cdot(p^2(\bx) \nabla u(\bx))= \lambda u(x) ,&\bx\in  \mathcal{M} \\
      \frac{\p u}{\p \bn}(\bx)=0,& \bx\in \p \mathcal{M},
\end{array}
\right.
\label{eqn:eigen_weight}
\end{equation}
where $p$ is the probability distribution.
The rate of convergence can also be obtained. This result will be reported in our subsequent paper.

\vspace{0.2in}
\noindent
{\bf Acknowledgments.}
This research was partial supported by NSFC Grant (11201257 to Z.S., 11371220 to Z.S. and J.S. and 11271011 to J.S.),
and  National Basic Research Program of China (973 Program 2012CB825500 to J.S.).

\bibliographystyle{abbrv}
\bibliography{poisson}

\end{document}